\documentclass[letterpaper,11pt]{scrartcl}

\usepackage{amsmath,amssymb,amsthm}
\usepackage{hyperref}
\usepackage[capitalize]{cleveref}
\usepackage{mathtools}
\usepackage{enumitem} 
\usepackage[capitalize]{cleveref}
\usepackage{bbm}
\usepackage{geometry}
\usepackage[utf8x]{inputenc} 
\usepackage[dvipsnames]{xcolor}
\usepackage{authblk} 
\setlist[enumerate,1]{label={(\roman*)}}
\usepackage{xparse} 

\usepackage[ruled,linesnumbered]{algorithm2e}

\usepackage{caption}
\usepackage{subcaption}
\usepackage{dsfont}

\usepackage{thmtools,thm-restate}

\newtheorem{theorem}{Theorem}

\newtheorem{corollary}[theorem]{Corollary}

\newtheorem*{theorem-non}{Theorem}

\newcommand{\abs}[1]{\left\lvert#1\right\rvert}
\newcommand*{\size}[1]{\left\lvert#1\right\rvert}

\NewDocumentCommand{\mathOrText}{m}
{%
	\ensuremath{#1}\xspace%
}

\DeclareDocumentCommand{\Pr}{m o}
{
	\mathOrText{ \mathds{P}\left[ #1 \IfNoValueF{#2}{\;\middle\vert\; #2} \right]}
}
\DeclareDocumentCommand{\E}{m o}
{
	\mathOrText{ \mathds{E}\left[ #1 \IfNoValueF{#2}{\;\middle\vert\; #2} \right]}
}
\DeclareDocumentCommand{\filtration}{o}
{
	\mathOrText{ \mathcal{F}\IfNoValueF{#1}{_{#1}}}
}

\newcommand{\bigO}[1]{\mathOrText{ \O\left(#1\right)}}
\newcommand{\smallO}[1]{\mathOrText{ o\left(#1\right)}}
\newcommand{\bigOmega}[1]{\mathOrText{ \Omega\left(#1\right)}}
\newcommand{\bigTheta}[1]{\mathOrText{ \Theta\left(#1\right)}}
\newcommand{\poly}[1]{\mathOrText{ \text{poly}\left(#1\right)}}
\newcommand*{\eulerE}{\mathrm{e}}

\usepackage[disable]{todonotes}
\usepackage{todonotes}

\def\N{\mathds{N}}
\def\R{\mathds{R}}
\newcommand{\1}{\mathds 1}
\def\Z{\mathds{Z}}
\def\O{\mathcal{O}}
\def\opt{\textsc{OPT}}

\def\alg1+1{(1+1)~EA\xspace}
\def\algBal{Balanced (1+1) EA\xspace}

\usepackage{tikz}
\usetikzlibrary{positioning}
\tikzset{blue node/.style={circle,fill=blue!20,draw,minimum size=0.55cm,inner sep=0pt},}
\tikzset{red node/.style={circle,fill=red!20,draw,minimum size=0.55cm,inner sep=0pt},}
\usepackage{float}

\title{Analysis of a Gray-Box Operator for Vertex Cover}

\author{Samuel~Baguley}
\author{Tobias~Friedrich}
\author{Timo~K{\"o}tzing}
\author{Xiaoyue~Li}
\author{Marcus~Pappik}
\author{Ziena~Zeif}
\affil{\normalsize
	\{firstname.lastname\}@hpi.de\\
	\vspace*{\baselineskip}
	Hasso Plattner Institute\authorcr
	University of Potsdam\authorcr
	Potsdam, Germany
}
\date{}

\begin{document}
	
	\maketitle
	\setcounter{page}{0}
	\thispagestyle{empty}

	\begin{abstract}
	  Combinatorial optimization problems are a prominent application area of evolutionary algorithms, where the \alg1+1 is one of the most investigated.
	  We extend this algorithm by introducing some problem knowledge with a specialized mutation operator which works under the assumption that the number of 1s of a solution is critical, as frequently happens in combinatorial optimization.
	  This slight modification increases the chance to correct wrongly placed bits while preserving the simplicity and problem independence of the \alg1+1.
	  
	  As an application of our algorithm we examine the vertex cover problem on certain instances, where we show that it leads to asymptotically better runtimes and even finds with higher probability optimal solutions in comparison with the usual \alg1+1.
	  Precisely, we compare the performance of both algorithms on paths and on complete bipartite graphs of size $n$.
	  Regarding the path we prove that, for a particular initial configuration, the \alg1+1 takes in expectation $\Theta(n^4)$ iterations while the modification reduces this to $\Theta(n^3)$, and present experimental evidence that such a configuration is reached.
	  Concerning the complete bipartite graph our modification finds the optimum in polynomial time with probability $1-1/2^{\Omega(n^\xi)}$ for every positive constant $\xi < 1$, which improves the known probability of $1-1/\text{poly}(n)$ for the \alg1+1.
	\end{abstract}

	\section{Introduction}

Evolutionary algorithms (EAs) are a diverse and highly versatile class of randomized search heuristics, and their effectiveness at solving combinatorial optimization problems has been recognized and widely studied since the nineties \cite[Chap.\ 14]{SMY02}. 
Typically, EAs are implemented with naïve mutation behavior, and adapt to a particular problem only via an associated fitness function.
This simplicity has lead to success in many applications, but leaves room for improvement on specific problems.

The minimum vertex cover problem is an ideal target for evolutionary algorithms because it is NP-complete \cite{Karp72}, 
and appears in a broad range of practical applications. 
Already in 1994 Bäck and Khuri \cite{BK94} defined the canonical
fitness function for EAs on vertex cover and provided experimental evidence that on some graph instances EAs can produce better approximate solutions than bespoke solvers. 
In the years that followed runtime analysis of EAs tended to be heuristic and experimental, but recently rigorous theoretical analyses 
have gained prominence.
The various approaches to this include proving runtime bounds for approximate solutions \cite{jansen2013approximating, oliveto2009analysis}, analyzing performance on best- or worst-case graph instances \cite{oliveto2007evolutionary, oliveto2008analysis}, 
and implementing genetic mechanisms like crossover \cite{Crossover16}. It is also common for new research to define variations of well-studied EAs, like Olivieto et al.'s \alg1+1 with an additional `diversity maintenance mechanism' \cite{oliveto2008analysis}, or a fitness function with multiple objectives \cite{friedrich2010approximating, kratsch2013fixed}.
 
In the application of evolutionary algorithms to real world problems, the full black-box view of the fitness function (ignoring all additional information about its inner workings) is at most the initial prototype; 
significant improvements can be made by applying problem knowledge for designing tailored variation operators \cite{PetersSAGLSUWKR19}. 
The resulting view is called gray-box optimization, and research tries to understand what general structural properties can be usefully employed in evolutionary computation \cite{WhitleyCG16}. 
We are interested in combinatorial problems where a certain subset has to be selected; typically, the number of elements in the chosen set is rather important. 
For example, a minimum spanning tree contains exactly $n-1$ edges, the vertex cover problem tries to minimize the number of vertices picked, and the independent set problem maximizes the selected vertices.
While \cite{LehreW12} argued that variation operators should be unbiased with respect to whether a $0$ or a $1$ encodes a specific property of a solution, the gray-box view of combinatorial problems wants the $1$ to mean `in' and $0$ to mean `out', and then design operators that can make use of this meaning. 
From the very general work of Rowe and Vose \cite{RoweV11} we can see that only for problems which have symmetrical meaning in $1$ and $0$ can the operators be assumed unbiased.
 
In this paper we augment
the \alg1+1 by allowing a second kind of mutation called a `balanced flip', in which a bit is chosen uniformly at random, and then swapped with the bit of a neighboring vertex 
of opposite value, if one exists.
At every iteration the augmented algorithm chooses between a balanced flip and a regular step of the \alg1+1 with equal probability.
The `balanced flip' behavior is already exhibited by the \alg1+1 on nearly-optimal bit-strings, the difference being that it takes an order of magnitude longer to perform each flip.
Thus the evolutionary algorithm presented here refines 
the \alg1+1 by incorporating its emergent behavior on inputs close to the optimum directly into the mutation operator; this is somewhat in the spirit of Giel and Wegener's `local \alg1+1' \cite{giel2003}. 
The immediate benefit of our augmentation is 
a tighter upper runtime bound for the \algBal on paths.

In \cref{sec: vertex cover path} we compare the runtime of the \alg1+1 solving vertex cover on paths to that of the \algBal. Our analysis of the \alg1+1 uses techniques 
similar to
those in \cite{giel2003} to obtain an expected runtime of $O(n^4)$, and we then apply those same techniques to the \algBal to prove an expected runtime of $O(n^3)$. We also give a partial proof of $\Omega(n^4)$ 
expected runtime for the \alg1+1, assuming that the algorithm at some point produces a bit-string with a long
consecutive sequence of incorrectly assigned bits. We provide experimental evidence to motivate this condition, and delay a full proof to future work.

In \cref{sec: vertex cover bipartite} we consider the complete bipartite graph $K_{L,R}$, which is a common instance class of interest when it comes to evolutionary algorithms.
Under the assumption that the ratio $R/L$ is at least $2$, the \alg1+1 is known to have exponential runtime with at least polynomial probability \cite[Theorem 5]{friedrich2010approximating}.
In contrast, the \algBal has far better performance on this graph instance, with exponentially small probability of not reaching the optimum in polynomial time.
	
	\section{Preliminaries}\label{sec: prelim}

In this article we demonstrate that incorporating even small amounts of problem knowledge can significantly improve reduce expected runtime of the \alg1+1 solving vertex cover. 
We provide a new evolutionary algorithm that we call \textit{\algBal} (Algorithm \ref{alg:balanced_EA}) that we compare with the classical \textit{\alg1+1} (Algorithm \ref{alg:1_plus_1}).

We introduce briefly the graph terminology used in this paper. 
A graph will be denoted $G=(V,E)$, and for any subgraph $H$ of $G$, we write $V(H)$ and $E(H)$ for vertices and edges of $H$ respectively.
We denote an edge $e=\left\{u,v\right\}\in E$ by $uv$ and the neighborhood of a vertex $v\in V$ in $G$ by $N(v)=\left\{u\in V\mid uv\in E\right\}$. 
Lastly, we define the length of a path $P=(V,E)$ to be number of vertices $|V|$.

A \emph{vertex cover} for a given graph $G=(V,E)$ is a subset of vertices $V' \subseteq V$ which satisfies
$V' \cap e \neq \varnothing$ for all $e \in E$.
A solution to the \emph{vertex cover problem}  is a vertex cover $V'$ which has minimal cardinality $|V'|$, in the sense that no cover of smaller cardinality exists. 
 
Depending on the context, we shall use $X$ to denote either the stochastic process induced by one of the algorithms or an arbitrary state of that process $X \in \{0, 1\}^n$.
We make sure that the meaning of $X$ is always clarified.
For $t \in \N$, $X_t$ denotes the configuration after $t$ iterations, and $X_0$ corresponds to the initial configuration of the process $X$.
Both algorithms studied in this work have search space $\{0,1\}^n$ of bit strings of length $n$, where $n = \abs{V}$.
If $V = (v_1,\dots,v_n)$ is some enumeration of the vertices of $G$, then a bit string $X = (x_1, \dots, x_n) \in \{0,1\}^n$ corresponds to a subset of $V_X \subseteq V$ via the mapping $v_i\in V_x \Leftrightarrow x_i=1$. 
We call a bit string $X$ \emph{feasible} if $V_X$ is a vertex cover for $G$, and \emph{infeasible} otherwise. If $V_X$ is a minimal vertex cover then $X$ is called \emph{optimal}.
Because the relationship between $X$ and $V_X$ is one-to-one, we will abuse notation and use $X$ to represent the bit string or subset of vertices interchangeably.
Letting $u(X) := \size{\{e \in E \mid V_X \cap e = \emptyset\}}$ be the number of uncovered edges in $X$, we have that $X$ is feasible if and only if $u(X) = 0$.
Further, we define $\abs{X}_{1} \coloneqq \abs{V_X}$ to be the number of ones in $X$, $\abs{X}_0 \coloneqq n - \abs{X}_{1}$ to be the number of zeroes, and $\opt = \abs{X_{\text{opt}}}_1$ to be the number of ones in the minimal vertex cover.

The \emph{fitness} of $X$ is given by the integer $f(X)$, where
$$
	f :\{0,1\}^n \to \N : X \mapsto |X|_1 + (n+1) u(X)
$$
defines the fitness function. The multiplicative factor $n+1$ ensures that the \alg1+1 and \algBal prioritise covering any uncovered edge over reducing the size of $X$, and that once covered an edge will never become uncovered. This fitness function (and slight variations of it) has been standard for evolutionary algorithms solving the vertex cover problem since \cite{BK94}.
With this in hand, the \alg1+1 can be defined.

\begin{algorithm}
	\caption{\alg1+1} \label{alg:1_plus_1}
	
	Choose $X \in \{0,1\}^n$ (if unspecified, uniformly at random) \label{alg_EA:initial_solution}
	
	\While{stopping criterion not met}
	{
		$Y \gets$ flip each bit of $X$ independently with probability $1/n$
		\label{alg_EA:sample step}
		
		\If{$f(Y) \leq f(X)$}
		{
			$X \gets Y$
		}
	}
\end{algorithm}

Because the \alg1+1 has positive probability of sampling any point of $\{0,1\}^n$ in step \ref{alg_EA:sample step}, regardless of the current state $X$, it almost surely produces a solution to the vertex cover problem in finite time.
A fundamental property of interest, and the main focus of this article, is the \emph{expected} time to find this solution.
If the optimum cover is unique then a naïve estimate based on this behavior yields an extremely poor expected runtime of $n^n$ for the \alg1+1. 
Worst-case analysis, for example in \cite[Theorem 5]{friedrich2010approximating}, suggests that no better can be expected for general $G$.

We now introduce a version of the \alg1+1 which chooses with equal probability between two sampling behaviors, one of which is the same as in the \alg1+1, while the other mimics behavior of the \alg1+1 on feasible bit strings. We call this second kind of behavior a `balanced flip', and the augmented algorithm is therefore called the \algBal.

\begin{algorithm}
	\caption{\algBal{}} \label{alg:balanced_EA}
	Choose $X \in \{0,1\}^n$ (if unspecified, uniformly at random) \label{alg_EA_balanced::initial_solution}
	
	\While{stopping criterion not met}
	{
		\label{alg_EA_balanced::while_loop}
		
		$p \gets$ pick a number u.a.r.~in the interval $(0,1)$
		
		\If{$p \leq 1/2$}
		{
			$Y \gets$ flip each bit of $X$ independently with probability $1/n$
		}
		\Else
		{
			$v \gets$ pick a vertex u.a.r~from $V(G)$
			
			$N_v \gets $ $\{v' \in N(v) \mid x_v \neq x_u\}$
			
			\If{$N_v \neq \varnothing$ }
			{
				$u \gets$ pick a vertex u.a.r~from $N_v$
				
				$Y \gets$ flip $x_u$ and $x_v$
			}
			\Else
			{
				\textbf{go to} line~\ref{alg_EA_balanced::while_loop}
			}
		}
		\If{$f(Y) \leq f(X)$}
		{
			$X \gets Y$
		}
	}
\end{algorithm}

\subsection{Feasible bit strings}

In order to analyze the expected runtimes of the \alg1+1 and \algBal, we use the following multiplicative drift theorem to bound the time it takes for $X$ to become feasible.

\begin{theorem}[Doerr et.~al~\cite{doerr2012multiplicative}]
	\label{theorem::multiplicative_drift}
	Let $(X_t, t\in\N)$ be a 
	stochastic process on
	$S \cup \{0\}$, where $S \varsubsetneq \R^+$ is a set positive numbers with minimum $s_{\min}$,
	and let $T = \inf \{t \in \N \mid X_t = 0\}$.
	Suppose there exists a real number $\delta > 0$ such that for all $s \in S$ and $t<T$, $\E{X_t - X_{t+1}}[X_t = s] \geq \delta s$. 
	Then, for all $s_0 \in S$, 
	\[
	\E{T \mid X_0=s_0} \leq \frac{\ln(s_0) - \ln(s_{\min})  + 1 }{\delta}.
	\]
\end{theorem}

\begin{restatable}{lemma}{restLemmaFeasible}
	\label{lemma::feasible}
	Let $G=(V,E)$ be a graph with $n = |V|$ vertices, and let $T$ be the first time that the \alg1+1 samples a feasible string.
	Then
	$
		\E{T} \le \eulerE n (\ln(n) + 1/2)
	$
	and for $k>0$, $\Pr{T > 2 \eulerE k n(\ln\left(n\right)+1/2)}\le 2^{-k}$.
	
	Further, if $S$ is the first time that the \algBal samples a feasible string, then
	$
		\E{S} \le 2 \eulerE n (\ln(n) + 1/2)
	$
	and for $k>0$, $\Pr{S > 4 \eulerE k n(\ln\left(n\right)+1/2)} \le 2^{-k}$.
\end{restatable}

\begin{proof}
	We start by analysing the \alg1+1.
	Let $X_t \in \{0,1\}^n$ denote 
	the state of the \alg1+1 after $t$ iterations, where $X_0$ is the initial random configuration.
	We shall apply \Cref{theorem::multiplicative_drift} to the process $u(X_t)$, $t\ge0$.
	For all $t \in \N$, let $F_t \subseteq E$ be the set of edges that are not covered by $X_t$ and let $U_t \coloneqq \bigcup_{e \in F_t} e$. 
	Note that the algorithm attempts to flip exactly one bit with probability $(1 - 1/n)^{n-1} \ge 1/\eulerE$.
	Whenever a single bit is flipped, this bit is chosen uniformly at random.
	Consider the (not necessarily induced) subgraph $(U_t, F_t)$ of $G$.
	If the flipped bit corresponds to a vertex $v \in U_t$, then $u(X_t) - u(X_{t+1})$ is equal to the degree of $v$ in $(U_t, F_t)$.
	Thus, given a uniformly random bit from $U_t$ is flipped, the expected difference $u(X_t)-u(X_{t+1})$ is the average degree of $(U_t, F_t)$, which is $2\abs{F_t}/\abs{U_t} = 2u(X_t)/\abs{U_t}$ by the hand shake lemma.
	Let us denote the event that exactly one bit in $U_t$ is flipped, and no other, by $C$. 
	Since $u(X_t)$ is non-increasing, $u(X_t) - u(X_{t+1})$ is non-negative.
	Therefore, we have
	\begin{align*}
		\E{u(X_t) - u(X_{t+1})}[u(X_t)] &\ge \E{(u(X_t) - u(X_{t+1}))\1_C}[u(X_t)]\\
		& \ge \frac {2u(X_t)}{\abs{U_t}}\Pr{C}[u(X_t)] \\
		&\ge \frac {2u(X_t)}{\abs{U_t}} \frac {\abs{U_t}} {\eulerE n}
		\ge \frac {2u(X_t)} {\eulerE n}.
	\end{align*}
	So the drift condition of \cref{theorem::multiplicative_drift} is satisfied by $u(X)$. Since $u(X) \in \{0,1,\dots, n^2\}$, we conclude that
	$\E{T} \le \eulerE n(\ln(n) + 1/2)$.
	
	To bound the expected time for the \algBal{} to find a feasible solution, note that with probability $1/2$ it performs the same actions as the \alg1+1.
	Moreover, the balanced flip operations cannot increase $u(X_t)$ due to the fitness function $f$.
	Thus we can essentially use the same proof as for the \alg1+1 but with drift $\E{u(X_t) - u(X_{t+1})}[u(X_t)] \ge u(X_t)/\eulerE n $, which yields an expected hitting time of at most $2 \eulerE n(\ln(n) + 1/2)$.
	
	In order to prove the second part of the statement, we start once again with the \alg1+1.
	We prove our claim via induction over $k \in \N$.
	The base case $k=0$ is trivial.
	Now, consider some $k > 0$ and assume the claim holds for $k - 1$.
	Let $A$ be the event that no feasible solution was found within the first $2 \eulerE (k-1) n \left(\ln(n) + 1/2\right)$ iterations and let $B$ be the event that no feasible solution is found within the last $2 \eulerE n \left(\ln(n) + 1/2\right)$ iterations.
	Then the desired probability can be written as $\Pr{A \text{ and } B} = \Pr{B}[A] \Pr{A}$.
	By our induction hypothesis, we have $\Pr{A} \le 2^{-(k-1)}$.
	Further, note that the bound on the expected number of iterations for finding a feasible solution, proven above, holds for arbitrary starting states.
	Thus, by Markov's inequality, we have $\Pr{B}[A] \le \frac{1}{2}$, which proves the upper bound.
	The proof for the \algBal{} is obtained analogously, using the respective expected hitting time.
\end{proof}

We frequently make use of \Cref{lemma::feasible} throughout different parts of this paper.
Note that a desirable property of the update criteria of both Algorithms~\ref{alg:1_plus_1} and~\ref{alg:balanced_EA} and the fitness $f$ is that, once a feasible solution is found, the algorithms never leave the space of feasible solutions again.

\subsection{Lower Bound Probability for Symmetric Random Walks}
The technical core of our proof for a lower bound on the running time for solving vertex cover on paths is the following general lemma.
It essentially lower bounds the hitting time of a process that is dominated by a symmetric random walk.
There are two main difference between \Cref{thm:lower_bound_symmetric_jump_walk} and most bounds that can be found in the literature. 
Firstly, the presented bound does not only hold in expectation but at least with constant probability and, secondly, the description of the process allows for arbitrarily large jumps, as long as their probability is bounded. 

\begin{restatable}{lemma}{lowerBoundSymmetricJumpWalk}
	\label{thm:lower_bound_symmetric_jump_walk}
	Let $d \in \N$ with $d \ge 4 \ln(10)$ and let $Z_t \in \N$ be a stochastic process such that, for some $p \in (0, 1]$ and $q \in [0, 1)$ with $q+2p \le 1$, it holds that
	\begin{enumerate}
		\item $\forall s \in \N \cap [0, d-1]: \Pr{Z_{t+1} = s+1}[Z_t = s] \le p$,
		\item $\forall s \in \N \cap [1, d-1]:$ 
		$\Pr{Z_{t+1} = s+1}[Z_t = s] \le \Pr{Z_{t+1} = s-1}[Z_t = s]$ and 
		\item $\forall s \in \N \cap [0, d]: \Pr{\abs{Z_{t+1} - Z_{t}} > 1}[Z_t = s] \le q$.
	\end{enumerate}
	Then, for $Z_0 = 0$ and $T = \inf\{t \in \N \mid Z_t \ge d\}$, it holds that
	\[
	\Pr{T \ge \min\left\{\frac{1}{5q}, \frac{d^2 (1-q)}{4 \ln(10) p}\right\}} \ge \frac{16}{25} .
	\qedhere
	\]
\end{restatable}

To prove \Cref{thm:lower_bound_symmetric_jump_walk}, we need two additional ingredients.
The first one is an elementary bound on the lower tail of a geometric random variable, which will allow us to lower bound the time until a certain event appears.  

\begin{restatable}{lemma}{lowerTailGeometric}
	\label{lemma:lower_tail_geometric}
	Let $X$ be a geometric random variable with success probability $p \in (0, 1)$.
	For all $c \in \R_{\ge 0}$ it holds that $\Pr{X \ge \frac{c}{p}} \ge 1 - c$.
\end{restatable}

\begin{proof}
	If $\frac{c}{p} \le 1$ then $\Pr{X \ge \frac{c}{p}} = 1 \ge 1-c$ holds trivially.
	Assume $\frac{c}{p} > 1$ and let $k = \left\lceil\frac{c}{p}\right\rceil \ge 2$.
	By the definition of a geometric random variable and Bernoulli's inequality, we have 
	\[
	\Pr{X \ge \frac{c}{p}} = \Pr{X \ge k} = \left(1 - p\right)^{k - 1} \ge 1 - p \cdot (k-1). 
	\]
	Observing that $\left\lceil\frac{c}{p}\right\rceil - 1 \le \frac{c}{p}$ concludes the proof.
\end{proof}

The second ingredient is a concentration result on the sum of independent geometric random variables, which is similar to the Chernoff bound.
More specifically, we are interested in obtaining a lower bound on such a sum.

\begin{restatable}{theorem}{chernoffGeometricRv}[{Doerr~\cite{Doerr2019}, Theorem $1.10.32$(b)}]
	\label{theorem:chernoff_geometric_rv}
	Let $X_1, \dots X_n$ be independent geometric random variables with common success probability $p>0$.
	Then, for $X= \sum_{i = 1}^{n} X_i$, $\mu = \frac{n}{p}$ and all $\delta \in (0, 1)$ it holds that
	\[
	\Pr{X \le (1-\delta) \mu} \le \eulerE^{- \frac{\delta^2}{2 - \frac{4}{3} \delta} n}.
	\qedhere
	\]
\end{restatable}

Given these two result, we are able to prove \Cref{thm:lower_bound_symmetric_jump_walk}.

\begin{proof}[Proof of \Cref{thm:lower_bound_symmetric_jump_walk}]
	We start by constructing a modified stochastic process $Z'_t \in D$ for $D = \Z \cap [-d, d + 1]$ with the following transition probabilities
	\begin{itemize}
		\item $\forall s \in D \setminus \{-d, d\}:$ $\Pr{Z'_{t+1} = s-1}[Z'_t = s] = \Pr{Z'_{t+1} = s+1}[Z'_t = s] = p$,
		\item $\forall s \in D \setminus \{-d, d\}: \Pr{Z'_{t+1} = d+1}[Z'_t = s] = q$ and
		\item $\forall s \in \{-d, d, d+1\}: \Pr{Z'_{t+1} = s}[Z'_t = s] = 1$
	\end{itemize}
	Set $Z'_0 = 0$ and $T' = \inf \{t \in \N \mid \abs{Z'_t} \ge d\}$.
	Note that due to \cref{thm:lower_bound_symmetric_jump_walk}(ii), $\abs{Z'_t}$ and $Z_t$ can be coupled in such a way that $\abs{Z'_t} \ge Z_t$ for all $t < T$ and $\abs{Z'_t} \ge d$ for all $t \ge T$. In particular, $T\ge T'$ almost surely.
	Consequently we have for all $t \in \N$ that $\Pr{T \ge t} \ge \Pr{T' \ge t}$. 
	
	We proceed by analyzing $T'$.
	To this end, we start by treating cases in which $Z'_t$ directly jumps from $s \in D$ with $\abs{s} \le d - 1$ to $d+1$ separately.
	We define a new stopping time $T_j = \inf\{t \in \N_{\ge 1} \mid \abs{Z'_{t-1}} < d  \text{ and } Z'_t = d+1\}$.
	Note that, for all $m \in \N$,  
	\begin{align*}
		\Pr{T' \ge m} &\ge \Pr{T' \ge m \text{ and } T_j \ge m} \\
		&= \Pr{T' \ge m}[T_j \ge m] \cdot \Pr{T_j \ge m} . 
	\end{align*}
	We set $m = \min\left\{\frac{1}{5q}, \frac{d^2 (1-q)}{4 \ln(10) p}\right\}$ and lower bound $\Pr{T_j \ge m}$ and $\Pr{T' \ge m}[T_j \ge m]$ separately.
	
	To obtain a lower bound on $\Pr{T_j \ge m}$, observe that, at any point in time $t \in \N$, $Z'_t$ has a probability of at most $q$ to do the desired jump.
	Thus, $T_j$ dominates a geometrically distributed random variable with success probability $q$.
	As $m \le \frac{1}{5q}$, \Cref{lemma:lower_tail_geometric} yields $\Pr{T_j \ge m} \ge \frac{4}{5}$.
	
	For lower bounding $\Pr{T' \ge m}[T_j \ge m]$, let $Y_t \in D$ be a stochastic process with 
	\begin{itemize}
		\item $\forall s \in D \setminus \{-d, d\}:$ \\
		$\Pr{Y_{t+1} = s-1}[Y_t = s] = \Pr{Y_{t+1} = s+1}[Y_t = s] = \frac{p}{1-q}$,
		\item $\forall s \in \{-d, d, d+1\}: \Pr{Y_{t+1} = s}[Y_t = s] = 1$.
	\end{itemize}
	Set $Y_0 = 0$ and $T_Y = \inf\{t \in \N \mid \abs{Y_t} \ge d\}$.
	By defining $Y$ in this way, and given that $T_j \ge m$, we can couple $Y_t$ and $Z'_t$ in such a way that $Y_t = Z'_t$ for all $t < m$.
	Therefore, we have $\Pr{T' \ge m}[T_j \ge m] = \Pr{T_Y \ge m}$.
	Further, note that $Y_t$ changes at most by $1$ in each step and has a symmetric probability to increase or decrease.
	We proceed to lower bound $\Pr{T_Y \ge m}$ in two steps.
	Firstly, we show that, with positive probability, $Y_t$ changes at least $\bigOmega{d^2}$ times to hit $\abs{Y_t} \ge d$.
	Secondly, we argue that this takes at least $m$ time steps.
	
	Define the random set $S = \{ t \in \N_{\ge 1} \mid Y_t \neq Y_{t-1}\}$ and let $N = \size{S}$.
	Note that $N$ is almost surely finite, as at any time $t$ the probability that the process increases $2d-Y_t$ times in a row is bounded below by $p^d/(1-q)^d$, after which $Y_t = d$ and the process remains constant.
	Let $\tau_1 < \tau_2 < \dots < \tau_N$ denote the ordered elements of $S$ and set $\tau_0 = 0$.
	Note that, as long as $\abs{Y_t} < d$, there is a non-zero probability for $Y_t$ to change in each step.
	Thus, we have $Y_{\tau_N} = d$ and $T_Y = \tau_N$ with probability $1$.
	Now, we define the process $Y'_i = Y_{\tau_{i \wedge N}}$ where $\tau_i \wedge N$ is shorthand for the minimum of $\tau_i$ and $N$.
	$Y'_i$ is a martingale with respect to the natural filtration over $Y_{t}$, as it increases or decreases with probability $\frac{1}{2}$ at each step if $\tau_i < N$ and remains constant otherwise.
	Using Azuma's inequality and the fact that the step size of $Y'_i$ is upper bounded by $1$, we obtain for $i \le \frac{d^2}{2 \ln(10)}$ that
	$
	\Pr{Y'_i \ge d} \le \eulerE^{- \frac{d^2}{2 i}} \le 1/10.
	$
	As almost surely $\abs{Y'_N} = \abs{Y_{\tau_N}} = d$, this yields $\Pr{N \le \frac{d^2}{2 \ln(10)}} \le 1/10$, meaning that $Y_t$ requires to change at least $\frac{d^2}{2 \ln(10)}$ times with probability at least $\frac{9}{10}$.
	It remains to bound the time that it takes $Y_t$ to change this often.
	To this end, note that $\tau_N = \sum_{i = 1}^{N} (\tau_{i} - \tau_{i-1})$.
	Recall that $T_Y = \tau_N$ almost surely.
	Setting $b=\frac{d^2}{2 \ln (10)}$ we obtain
	\begin{align*}
		\Pr{T_Y \ge m} 
		&= \Pr{\sum_{i = 1}^{N} \tau_{i} - \tau_{i-1} \ge m} \\
		&\ge \Pr{\sum_{i = 1}^{N} \tau_{i} - \tau_{i-1} \ge m \text{ and } N \ge b} \\
		&= \Pr{\sum_{i = 1}^{N} \tau_{i} - \tau_{i-1} \ge m}[N \ge b] \cdot \Pr{N \ge b} \\
		&\ge \Pr{\sum_{i = 1}^{b} \tau_{i} - \tau_{i-1} \ge m} \cdot \frac{9}{10} .
	\end{align*}
	Next, observe that each of the differences $\tau_{i} - \tau_{i-1}$ independently follows a geometric distribution with success probability $\frac{p}{(1-q)}$.
	Using \Cref{theorem:chernoff_geometric_rv} with $n = b=\frac{d^2}{2 \ln (10)}$, $\mu = \frac{d^2 (1-q)}{2 \ln (10) p}$ and $\delta = \frac{1}{2}$ we have 
	\[
	\Pr{\sum_{i = 1}^{b} \tau_{i} - \tau_{i-1} < m} 
	\le \Pr{\sum_{i = 1}^{b} \tau_{i} - \tau_{i-1} < \frac{d^2 (q-1)}{4 ln(10) p}}
	\le \frac{1}{10} 
	\]
	for $d \ge 4 \ln(10)$.
	
	Therefore, we get $\Pr{T' \ge m}[T_j \ge m] \ge \Pr{T_Y \ge m} \ge \frac{9}{10} \cdot \frac{9}{10} > \frac{4}{5}$ and $\Pr{T' \ge m} \ge \frac{4}{5} \cdot \frac{4}{5} = \frac{16}{25}$, which concludes the proof.
\end{proof}

	\section{Vertex Cover on Paths}\label{sec: vertex cover path}
Although the worst-case expected runtime of the \alg1+1 is known to be exponential in $n$,
\cite[Theorem 5]{friedrich2010approximating}, far better runtime bounds can be obtained on simple graph instances. 
In this section we take $G$ to be a path of length $\abs{V} = n$, with the $i$-th bit of $X$ corresponding to the $i$-th vertex of $G$, numbering sequentially along the path.
By considering path instances, we can precisely analyse the runtime of the \alg1+1, and then transfer that analysis to the \algBal, showing that in expectation it performs better by a linear factor.
It is reasonable to expect that the \algBal also improves on the runtime of the \alg1+1 on other graph instances, but rigorous analysis of those will require more general tools than those we develop below.

To achieve the lower bound for either algorithm, we require two additional assumptions, the first of which is defined in more detail in \cref{sec:lower_bound_path} and supported by experimental evidence in \cref{fig:EA_bad_paths}.
\begin{enumerate}[label={(\Alph*)}]
	\item\label{Adelai} With probability $\ge 1/2$, there exists a $t\ge0$ such that $X_t$ is feasible and contains a connected subpath of vertices which are not in their optimal state with size $\bigTheta{n}$.
	\item\label{very odd} $\abs{V} = n$ is odd. 
\end{enumerate}
Equivalently to \ref{Adelai}, we could assume that $X$ has an initial distribution which is supported by the subset of $\{0,1\}^n$ containing only feasible bit strings.
Assumption \ref{Adelai} simply ensures that $X$ hits this subset with probability $1/2$, and since $X$ is a strong Markov process, restarting at this hitting time preserves the lower bound.
Assumption \ref{very odd} is needed for particular arguments in \cref{sec:lower_bound_path}, but seems unlikely to be necessary.

The two main theorems of this section are as follows.
\begin{restatable}{theorem}{restPriceBoundClassic}
	\label{theorem::price_bound_1_plus_1}
	Let $P=(V,E)$ be a path 
	of length $n=|V|$.
	In expectation, the \alg1+1 samples an optimal solution in $\O(n^4)$ iterations, and in $\Omega(n^4)$ iterations if assumptions \ref{Adelai} and \ref{very odd} hold true.
\end{restatable}

\begin{restatable}{theorem}{restPriceBoundBalanced}
	\label{theorem::price_bound_balanced}
	Let $P=(V,E)$ be a path of length $n=|V|$.
	In expectation, the \algBal samples an optimal solution in $\O(n^3)$ iterations, and in $\Omega(n^3)$ iterations if assumptions \ref{Adelai} and \ref{very odd} hold true.
\end{restatable}
	
	\subsection{Upper Bound Running Time}\label{sec:upper_bound_path}

By \Cref{lemma::feasible}, we know that the \alg1+1 finds a feasible solution in expectation in $\bigO{n \log(n)}$ iterations, given any starting distribution, including in particular the uniform distribution.
Once a feasible solution is found, all states in following iterations are feasible as well.
Hence our analysis can focus on the expected runtime given a feasible initial state.
If $n$ is odd, there exists a unique optimal vertex cover of size $\opt = (n-1)/2$, and if $n$ is even, there are multiple optimal covers of size $\opt = n/2$. 
We define the \emph{level} of $X$ at time $t$ to be $\abs{X_t}_1 - \opt$, to parametrize the deviation of $X$ from the optimum.
Our strategy is to compute the expected number of iterations that $X$ spends at any level $\ell$, and then sum over all levels to obtain an upper bound for total runtime.
For this purpose, we define $Y_\ell$ to be the total number of iterations spent at level $\ell\ge0$, i.e.~$
Y_\ell = \sum_{t=0}^Y \mathbf 1\{\abs{X_t}_1 = \ell + \opt \}$,
and let
$Y = \sum_{i=1}^{\lceil n/2 \rceil} Y_i $
be the first hitting time of level $\ell=0$.

The main work of this section is in proving the following bounds on the expected time spent at level $\ell$ - which may be zero - by the \alg1+1 and \algBal respectively.

\begin{restatable}{lemma}{expectedTimeOneLevelEA}
\label{lemma::expected_time_one_level}
The expected number of iterations that the \alg1+1 spends at level $\ell$ is of order $O(n^4/ \ell^2)$.
\end{restatable}

\begin{restatable}{lemma}{expectedTimeOneLevelBal}
\label{lemma::expected_time_one_level_bal}
The expected number of iterations that the \algBal spends at level $\ell$ is of order $O(n^3/ \ell^2)$.
\end{restatable}
With these results in hand, we prove the following upper bound on the number of expected iterations before the \alg1+1 samples an optimal solution for the first time.

\begin{restatable}{theorem}{restThmUpperBoundPath}
	\label{theorem::upper_bound_1_plus_1_VC}
	Let $P=(V,E)$ be a path of length $n=|V|$.
	Then the expected number of iterations that the \alg1+1 needs to sample an optimal solution is of order $\O(n^4)$, and the expected number of iterations that the \algBal needs to sample an optimal solution is of order $\O(n^3)$.
\end{restatable}
\begin{proof}
	In either case,
	by \cref{lemma::feasible} $X$ becomes feasible for the vertex cover after an expected $\O(n\log(n))$ iterations.
	
	At this point the algorithm has an expected remaining runtime of
	$$
	\E{Y}
	= \mathds E\Big[\sum_{\ell = 1}^{\left \lceil \frac{n}{2} \right \rceil} Y_\ell \Big] 
	= \sum_{\ell = 1}^{\left \lceil \frac{n}{2} \right \rceil} \E{Y_\ell}
	$$ 
	which according to \cref{lemma::expected_time_one_level} is of order $\O(n^4)$ for the \alg1+1, and by \cref{lemma::expected_time_one_level_bal} is $\O(n^3)$ for the \algBal, since in either case $\sum_{\ell = 1}^{\left \lceil \frac{n}{2} \right \rceil} 1/\ell^2$ is bounded above by a constant independent of $n$.
\end{proof}

\paragraph{A Random Walk Coupling:}\label{sec:coupling}
The proofs of Lemmas \ref{lemma::expected_time_one_level} and \ref{lemma::expected_time_one_level_bal} make use of a coupling argument, given below, between 
$X$ 
and an independent random walk on a path of length $\Theta(n)$. To take advantage of that coupling we prove the following Lemma; 
an alternative proof of this result using drift theory can be found in G{\"o}bel et al.\ \cite[Theorem 13]{DBLP:journals/corr/abs-1806-01919}. 

\begin{restatable}{lemma}{hittingTimeOneBarrierProcess}
\label{lemma::hitting_time_one_barrier_process}
Let $q\in(0,1/2]$, and let $Z$ be the symmetric Markov chain on $\{0,1,\dots,d\}$ with transitions $P(r,s) = \Pr{Z_{t+1}=r | Z_t = s}$ defined by
$$
	P(s, s-1) =
	\begin{cases}
		q\;\ \text{ if } 1\le s \le d-1,\\
		2q \text{ if } s=d,\\
		0\;\ \text{ if } s=0,
	\end{cases}
	\hspace{-0.5em}
	P(s, s+1) =
	\begin{cases}
		q  \text{ if } 1\le s \le d-1,\\
		0 \text{ if } s\in\{0,d\},
	\end{cases}
$$
and $P(s, s) = 1 - P(s, s-1) - P(s, s+1)$. 
This is the symmetric random walk with reflecting barrier at $d$ and absorbing barrier at $0$, with an additional probability to stay put at any vertex. Suppose $Z_0 = d$ and $T = \min\{t\ge0 : Z_t = 0\}$ be the first time that $Z$ hits 0. Then $\E{T} \le d^2/2q$. 
\end{restatable}
\begin{proof}
	Let $Y$ be the symmetric random walk on $\{0,1,\dots,2d\}$ started at $Y_0 = d$, with absorbing barriers $0$ and $2d$. 
	It is well known - see for example Cox and Miller \cite[Exercise 2.12]{CoxMiller} - that the expected hitting time of $\{0,2d\}$ by $Y$ is $\E{\min\{t\ge0 : Y_t \in\{0,2d\}} = d^2$.
	
	If instead $Y$ moves left or right with probability $q$ and stays put with probability $1-2q$ then $Y$ sits for an independent geometrically distributed time $S$ at each vertex before moving, so Wald's equation gives $\E{\min\{t\ge0 : Y_t \in\{0,2d\}}= d^2\E{S} = d^2/2q$.
	Since $Z$ has the same law as $d-\abs{Y}$, it has expected runtime $\E{T} = d^2/2q$.
\end{proof}

The unique optimal solution on the odd-length path - which will be the terminal value of $X$ - has alternate vertices selected, starting with the second and ending with the $(n-1)$-th.
Similarly, both optima of the even-length path are similarly alternating, apart from a possible pair of adjacent ones.
Let us suppose that $X$ is at level $\ell$ and is feasible. 
We shall begin by supposing that $X$ is irreducible, and bound the time it takes to become reducible. 
Since $X$ is feasible, no two neighboring vertices are unselected, and since $X$ is irreducible, there is no subpath of consecutive ones of length more than $2$. 
An example for $n=11$ is given below.
\begin{figure}[H]
\centering
\begin{tikzpicture}
	\node[blue node, label = {1}] (1) {$0$};
	\node[blue node, label = {2}] (2) [right = 0.7em of 1] {$1$};
	\node[blue node, label = {3}] (3) [right = 0.7em of 2] {$1$};
	\node[blue node, label = {4}] (4) [right = 0.7em of 3] {$0$};
	\node[blue node, label = {5}] (5) [right = 0.7em of 4] {$1$};
	\node[blue node, label = {6}] (6) [right = 0.7em of 5] {$0$};
	\node[blue node, label = {7}] (7) [right = 0.7em of 6] {$1$};
	\node[blue node, label = {8}] (8) [right = 0.7em of 7] {$0$};
	\node[blue node, label = {9}] (9) [right = 0.7em of 8] {$1$};
	\node[blue node, label = {10}] (a) [right = 0.7em of 9] {$1$};
	\node[blue node, label = {11}] (b) [right = 0.7em of a] {$0$};

	\path[draw,thick]
	(1) edge node {} (2)
	(2) edge node {} (3)
	(3) edge node {} (4)
	(4) edge node {} (5)
	(5) edge node {} (6)
	(6) edge node {} (7)
	(7) edge node {} (8)
	(8) edge node {} (9)
	(9) edge node {} (a)
	(a) edge node {} (b);
\end{tikzpicture}
\end{figure}
Here $X$ is at level $\ell = 1$, although in this example most nodes (from $3$ to $9$ inclusive) are not in their optimal state.
Now we can define a dual process $\tilde X$ on the path of length $2 + (n-1)/2$, which is coupled to $X$, in the following way: for $i\in\{1,\dots,(n-1)/2\}$,
\begin{align*}
	\tilde X(i) & = X(2i)\cdot \1\{\text{at least one of $X(2i-1)$ and $X(2i+1)$ is a 1}\}\\
	& = X(2i)(X(2i-1) + X(2i+1) - X(2i-1)X(2i+1)),
\end{align*}
and $\tilde X(0) = X(1)$, $\tilde X(n+1/2) = X(n)$.
Then the dual state to the one above is
\begin{figure}[H]
\centering
\begin{tikzpicture}
	\node[red node, label = {0}] (0) {$0$};
	\node[red node, label = {1}] (1) [right = 0.7em of 0] {$1$};
	\node[red node, label = {2}] (2) [right = 0.7em of 1] {$0$};
	\node[red node, label = {3}] (3) [right = 0.7em of 2] {$0$};
	\node[red node, label = {4}] (4) [right = 0.7em of 3] {$0$};
	\node[red node, label = {5}] (5) [right = 0.7em of 4] {$1$};
	\node[red node, label = {6}] (6) [right = 0.7em of 5] {$0$};

	\path[draw,thick]
	(0) edge node {} (1)
	(1) edge node {} (2)
	(2) edge node {} (3)
	(3) edge node {} (4)
	(4) edge node {} (5)
	(5) edge node {} (6);
\end{tikzpicture}
\end{figure}
If $X$ is the \alg1+1, then conditional on $X$ flipping at most two bits at a time,%
\footnote{Since $X$ is feasible and irreducible, it cannot flip only one bit in a step.}
 the `particles' at positions $1$ and $5$ move independently like symmetric random walks (with reflective boundaries) until the time that they are adjacent, at which point $X$ is reducible.
If $X$ is the \algBal, then the same random walk behavior is observed, but with differing transition probabilities.
In this example there are only two such particles, since $X$ is at level $\ell=1$, but in general there will be $2\ell$. 
As long as $X$ is not optimal, we define the minimum distance (measured by number of zero nodes) between any two particles in $\tilde X$ to be the longest subpath of consecutive zeroes (corresponding to subpaths of alternating bits in $X$).
This is itself a stochastic process on $\{0,\dots,d\}$, where 
$d\le 2\ell$.
For ease of notation we denote this process by $H$.

\paragraph{The \alg1+1 - Proof of \cref{lemma::expected_time_one_level}:}
Take $X$ to be the \alg1+1, and for the moment, let us condition on the event that the algorithm only flips 2 bits at every iteration. 
Under this conditioning, the process $H$ can only make steps of size one in either direction. It need not be symmetric, but decreases (that is, $H_{t+1} = H_{t-1}$) with probability $\ge 2p^2(1-p)^{n-2}$,
and increases with probability $\le 2p^2(1-p)^{n-2}$.
Thus the expected number of iterations taken for $H$ to hit zero is less than the expected hitting time $T$ of zero by the independent process $Z$ on $\{0,1,\dots,d\}$ satisfying the conditions of \cref{lemma::hitting_time_one_barrier_process} with $q = 2p^2(1-p)^{n-2}$.
Recalling that $p = 1/n$ and $d\le n/2\ell$, the following corollary is immediate.
\begin{corollary}\label{corollary:iodine}
$\E{T} \le d^2/4p^2(1-p)^{n-2} \le n^4/\ell^2$.
\end{corollary}
Thus we have an upper bound on the expected time for $X$ to become reducible, conditional on the \alg1+1 flipping no more than 2 bits at every step.
Now we want to remove that conditioning.
To do so, we modify the transitions of $Z$ so that it returns to position $d$ with probability $4p^4$ at every step; this is an upper bound for the probability that $(H_{t+1} - H_t)\ge 2$, that is, that $H$ increases by 2 or more in a single step.
Let us denote this modified version of $Z$ by $\tilde Z$.
The transitions of $\tilde Z$ are
$$
	\tilde  Z_{t+1} =
	\begin{cases}
		\tilde Z_t + 1 \qquad &\text{with probability }2p^2(1-p)^{n-2},\\
		\tilde Z_t -1 &\text{with probability }2p^2(1-p)^{n-2},\\
		d &\text{with probability } 4p^4,\\
		\tilde Z_t &\text{otherwise},
	\end{cases}
$$
when $\tilde Z_t\in\{1,\dots,d-1\}$, and similar in the obvious way for $\tilde Z_t=d$. 
In essence, $\tilde Z$ is `restarted' every time $H$ increases by 2 or more, and so the expectation of the hitting time $\tilde T = \min\{t\ge0 : \tilde Z_t = 0\}$ is greater than the expected time it takes for $H$ to hit zero - that is, for $X$ to become reducible.

The question to answer now is how many times $\tilde Z$ will have to restart; we shall show that it is constant in expectation.
Each `run' of $\tilde Z$, starting at position $d$ and ending when it hits zero or is restarted, is independent from the others. 
They are also all indentically distributed, with the same law as $T$ from \cref{corollary:iodine}.
The number of restarts is therefore a geometric random variable, say $R$, and to calculate its mean we need only calculate the probability of success, that is, of $\tilde Z$ hitting zero before it is sent back to $d$.

The `restart time' of $\tilde Z$ is itself a geometric random variable $S$, independent of the position of $\tilde Z$, with success probability $4p^4$. 
Therefore, for any $r\in \N^+$,
$$
	\Pr{T < S} \ge \Pr{T \le r < S} = \Pr{T\le r} (1-4p^4)^r.
$$
We saw in \cref{corollary:iodine} that $\E{T} \le n^4/\ell^2$, and it follows from Markov's inequality that $\Pr{T\le 2\E{T}} \ge 1/2$. 
Therefore, taking $r = 2n^4/\ell^2$, we have that
$$
	\Pr{T<S} \ge \frac {(1-4n^{-4})^{2n^4/\ell^2}} 2 \to c\in(0,1) \quad \text{as }n\to\infty.
$$
Then Wald's equation yields that
$$
	\E{\tilde T} \le \E{T}\E{R} = \frac {n^4} {\ell^2\Pr{T\ge S}} = \frac {n^4} {(1-c)\ell^2}.
$$
This gives an upper bound for the expected number of iterations it takes for $H$ to hit zero - that is, $X$ to become reducible - given that $X$ started in level $\ell$.

It remains to note that when $X$ is in a reducible state, it has a positive probability of taking a step to an irreducible state but remaining at level $\ell$. However, the probability of this occurring before $X$ is reduced is bounded above by a constant,%
\footnote{This requires two simultaneous bit flips, and roughly has probability $\ell/n^2$, compared to $1/n$ for simply reducing. Markov's inequality gives the constant.} 
and so an argument similar to above using Wald's equation and restarting gives that $\E{Y_\ell}\le kn^4 /\ell^2$ for some constant $k>1$; this proves \cref{lemma::expected_time_one_level}.

\paragraph{The {\algBal} - Proof of \cref{lemma::expected_time_one_level_bal}:}
The same arguments as for the \alg1+1 hold for the \algBal, with the difference that $H$ decreases with larger probability. More precisely, with probability $1/2$ the algorithm chooses a 
bit uniformly at random, and attempts to swap it with a neighbor of a different value.
There are 4 balanced flips out of at most $2n$ which cause $H$ to decrease;
conditioning the algorithm on flipping no more than 2 bits in a single step, this leads to a decrease probability of
$$
	q \ge \frac 1 2 \Big( \frac 4 {2n} + 4p^2(1-p)^{(n-2)} \Big),
$$
which again via \cref{lemma::hitting_time_one_barrier_process} leads to a runtime bound of $d^2/2q \le n^3/\ell^2$. The same restarting arguments as above then yield the result of \cref{lemma::expected_time_one_level_bal}.

	\subsection{Lower Bound Running Time}\label{sec:lower_bound_path}

In this section, we prove the lower bound parts of Theorems \ref{theorem::price_bound_1_plus_1} and \ref{theorem::price_bound_balanced}.
That is, given a suitably chosen starting configuration, the \alg1+1 and \algBal require at least $\bigOmega{n^4}$ iterations to find the minimum vertex cover on a path of odd length $n$.
Again we note that such instances $G$ only have one minimal vertex cover, namely the set of vertices at even positions, which simplifies our analysis.

We start by introducing some additional notation and terminology that will come in handy for stating and proving the main result of the section.
Let $P$ be a path of odd length $n$ with vertices $v_1, \dots, v_n$ and let $X \in \{0, 1\}^n$ be a bit string $x_1, \dots, x_n$, representing a solution candidate for the vertex cover problem.

Similarly to \cref{sec:coupling}, a major role is played in these proofs by connected subpaths of $X$ in which all vertices are not in their optimal state, but they appear in a different form, and so we introduce some new notation.
Let $B(X) \subseteq [n]^2$ denote the set of all tuples $(i, j)$ with $i \le j$ that corresponds to endpoints of subpath $\{v_i, \dots, v_j\}\subseteq V$ such that all vertices $v_k$ for $i \le k \le j$ are not in their optimal state.
That is, every $x_k$ for $i \le k \le j$ has value $0$ if and only if $k$ is even.
We shall call such subpaths `\emph{bad paths}'.
Then $X$ represents a minimum vertex cover if and only if $X$ is feasible and $B(X) = \emptyset$.
Assumption \ref{Adelai} is equivalent to assuming that for initial configuration $X_0$, $B(X_0)$ contains exactly on tuple $(i,j)$ with $j-i \in \bigTheta{n}$.

The main results of this section are the following two statements.

\begin{restatable}{theorem}{lowerBoundEAPath}
	\label{theorem:lower_bound_EA_path}
	Let $P=(V,E)$ be a path 
	of odd length $n=|V|$.
	Assume a feasible initial configuration $X_0$ with $B(X_0) =  \{(i, j)\}$ such that $j-i \in \bigTheta{n}$. 
	With constant positive probability the \alg1+1 
	requires at least $\bigOmega{n^4}$ iterations to find the optimal solution.
	Consequently, the expected number of iterations is in $\Omega(n^4)$.
\end{restatable}

\begin{restatable}{theorem}{lowerBoundBalEApath}
	\label{theorem:lower_bound_bal_EA_path}
	Let $P=(V,E)$ be a path 
	of odd length $n=|V|$.
	Assume a feasible initial configuration $X_0$ with $B(X_0) =  \{(i, j)\}$ such that $j-i \in \bigTheta{n}$. With constant positive probability the \algBal 
	requires at least $\bigOmega{n^3}$ iterations to find the optimal solution.
	Consequently, the expected number of iterations is in $\Omega(n^3)$.
\end{restatable}

\paragraph{The \alg1+1 - Proof of \cref{theorem:lower_bound_EA_path}:}
As long as $B(X_t) \neq \emptyset$, let $l(X_t)$ and $r(X_t)$ denote the left and right endpoint of the bad path in $P$ with respect to $X_t$ (i.e., $B(X_t) = \{(l(X_t), r(X_t))\}$).
Let $m = l(X_0) + (r(X_0) - l(X_0))/2$ and set $l(X_t) = r(X_t) = m$ for all $t \in \N$ with $B(X_t) = \emptyset$.
Define $b(X_t) = r(X_t) - l(X_t)$ and observe that $T = \inf \{t \in \N \mid b(X_t) \le 0\}$ is a lower bound on the required number of iterations to reach the optimum.

We proceed by showing that $T \in \Omega(n^4)$ with positive probability.
To this end, define two new stopping times $T_l = \inf \{t \in \N \mid l(X_t) \ge m - 4\}$ and $T_r = \inf \{t \in \N \mid r(X_t) \le m + 4\}$.
Observe that $T \ge \min \{T_l, T_r\}$.
Thus, lower bounding $T_l$ and $T_r$ with positive probability suffices to obtain the desired result for $T$.

The following lemma
characterizes the transition probabilities
of $l(X_t)$ and $r(X_t)$. 
Applying it in combination with \Cref{thm:lower_bound_symmetric_jump_walk}, we are able to prove \Cref{theorem:lower_bound_EA_path}.

\begin{restatable}{lemma}{endpointTransitions}
	\label{lemma:endpoint_transitions}
	For all $t < \min\{T_l, T_r\}$ it holds that 
	\begin{itemize}
		\item $\Pr{\abs{l(X_{t+1}) - l(X_{t})} > 2}[X_t] \le 2/n^4$,
		\item $\Pr{l(X_{t+1}) = l(X_{t}) + 2}[X_t] \le 1/n^2$ and
		\item $\Pr{l(X_{t+1}) = l(X_{t}) + 2}[X_{t}] = \Pr{l(X_{t+1}) = l(X_{t}) -2}[X_{t}]$ whenever $l(X_t) \ge 3$
	\end{itemize}
	and analogously 
	\begin{itemize}
		\item $\Pr{\abs{r(X_{t+1}) - r(X_{t})} > 2}[X_t] \le 2/n^4$,
		\item $\Pr{r(X_{t+1}) = r(X_{t}) - 2}[X_t] \le 1/n^2$ and
		\item $\Pr{r(X_{t+1}) = r(X_{t}) + 2}[X_{t}] = \Pr{r(X_{t+1}) = r(X_{t}) -2}[X_{t}]$ whenever $r(X_t) \le n - 2$.
	\end{itemize}
\end{restatable}

\begin{proof}
	First note that we assume $X_0$ to be feasible. 
	Thus, $X_t$ is also feasible for all $t \in \N$.
	Moreover, in the initial configuration $X_0$ we have exactly one bad path from position $l(X_0)$ to $r(X_0)$.
	This means, for every $k \in [n] \cap [l(X_0), r(X_0)]$, the bit at position $k$ is $1$ in $X_0$ if and only if $k$ is odd.
	On the other hand, for all $k \in [n] \setminus [l(X_0), r(X_0)]$, the bit at position $k$ is $1$ in $X_0$ if and only if $k$ is even.
	Thus, the only bit flips that the \alg1+1 can do is 
	\begin{itemize}
		\item flip an even length bit sequence either starting at $l(X_0)$ or ending at $l(X_0)-1$,		
		\item flip an even length bit sequence either ending at $r(X_0)$ or starting at $r(X_0)+1$
		\item or do both of the above at the same time.
	\end{itemize}  
	As long as this leads to a state $X_1$ with $B(X_1) \neq \emptyset$, these properties remain in place.
	Thus, inductively, the same holds for all $t < \min\{T_l, T_r\}$.
	
	Given this, we prove the claimed transition probabilities for $l(X_t)$.
	The transitions of $r(X_t)$ are analyzed analogously.
	Note that for $t < \min\{T_l, T_r\}$ we have $r(X_t) - l(X_t) \ge 8$.
	Thus, changing $l(X_t)$ by more than $2$ requires at least flipping either the bits $l(X_t)-4$ to $l(X_t)-1$ or $l(X_t)$ to $l(X_t)+3$, each of which happens with probability at most $\frac{1}{n^4}$.
	By union bound we get $P[\abs{l(X_{t+1}) - l(X_{t})} > 2 | X_t] \le 2/n^4$.
	Similarly, in order to increase $l(X_t)$ by exactly $2$, the bits $l(X_t)$ and $l(X_t) + 1$ need to be flipped at once.
	Thus, it holds that $P[l(X_{t+1}) = l(X_{t}) + 2 | X_t] \le 1/n^2$.
	Finally, consider $X_t$ such that $l(X_t) \ge 3$ and assume a set of bits $S \subseteq [n]$ is flipped, such that $l(X_t)$ changes by exactly $2$.
	If this succeeds, either $\{l(X_t)-2, l(X_t)-1\} \subseteq S$ or $\{l(X_t), l(X_t)+1\} \subseteq S$ but not both.
	Assume $\{l(X_t)-2, l(X_t)-1\} \subseteq S$, meaning that $l(X_t)$ decreases by $2$, and construct the set $S' = (S \setminus \{l(X_t)-2, l(X_t)-1\}) \cup \{l(X_t), l(X_t)+1\}$.
	Note that, if flipping the bit set $S$ succeeds, flipping $S'$ would also succeed as both result in the same fitness value.
	Moreover, we have $\size{S} = \size{S'}$, which means that both operations have the same probability.
	Finally, observe that the construction of $S'$ from $S$ is reversible.
	Thus, the overall probability of decreasing $l(X_t)$ by $2$ is the same as increasing it by $2$. 
\end{proof}

Using the characterization of the transition probabilities given in \Cref{lemma:endpoint_transitions} we use \Cref{thm:lower_bound_symmetric_jump_walk} to prove \Cref{theorem:lower_bound_EA_path}.

\begin{proof}[Proof of \Cref{theorem:lower_bound_EA_path}]
	Let $l(X_t), r(X_t)$, $b(X_t)$, $m$, $T$, $T_l$ and $T_r$ be defined as above.
	As discussed earlier, $T$ is a lower bound on the number of iterations that the \alg1+1 requires to get to the optimum.
	Furthermore, it holds that $T \ge \min\{T_l, T_r\}$.
	Thus, for every $t \in \N$, it holds that 
	\[
	\Pr{T \ge t} \ge 1 - \Pr{T_l < t \text{ or } T_r < t} \ge 1 - \Pr{T_l < t} - \Pr{T_r < t}.
	\]
	
	We proceed by using \Cref{thm:lower_bound_symmetric_jump_walk} to prove that, if $r(X_0) - l(X_0) \in \bigTheta{n}$, then there is some $\tau \in \bigTheta{n^4}$ such that $\Pr{T_l \ge \tau} \ge \frac{16}{25}$ and $\Pr{T_r \ge \tau} \ge \frac{16}{25}$.
	
	To this end, we consider the process $Z_t = \frac{l(X_t) - l(X_0)}{2} \cdot \1_{\{l(X_t) \ge l(X_0)\}}$ and set $d = \left\lfloor \frac{r(X_0) - l(X_0)}{4}\right\rfloor - 2$. 
	Note that $d \in \bigTheta{n}$, which implies $d \ge 4 \ln(10)$ for $n$ sufficiently large.
	By \Cref{lemma:endpoint_transitions}, we know that $Z_t$ satisfies the requirements of \Cref{thm:lower_bound_symmetric_jump_walk} for $p = \frac{1}{n^2}$ and $q=\frac{2}{n^4}$.
	Thus, for $\tau = \min \left\{ \frac{2}{5} n^4, \frac{d^2 n^2}{8 \ln(10)} \left(1 - \frac{2}{n^4}\right)\right\} \in \bigTheta{n^4}$ and $T_d = \inf\{ t \in \N \mid Z_t \ge d\}$, \Cref{thm:lower_bound_symmetric_jump_walk} yields $\Pr{T_d \ge \tau} \ge \frac{16}{25}$.
	Now, note that for our choice of $Z_t$ and $d$, $l(X_t) \ge m - 4$ implies $Z_t \ge d$.
	Therefore, we have $\Pr{T_l \ge \tau} \ge \frac{16}{25}$.
	
	For $r(X_t)$, we can argue analogously that $\Pr{T_r \ge \tau} \ge \frac{16}{25}$ for the same $d \in \bigTheta{n}$ and $\tau \in \bigTheta{n^4}$ as above and using the process $Z_t = \frac{r(X_0) - r(X_t)}{2} \cdot \1_{\{r(X_t) \le r(X_0)\}}$.
	Thus, we obtain $\Pr{T \ge \tau} \ge \frac{7}{25}$, which proves the first part of the statement.
	For the second part, note that by $T \ge 0$ and the law of total expectation $\E{T} \ge \tau \cdot \frac{7}{25} \in \bigTheta{n^4}$.
	By monotonicity of the expectation, this carries over to the expected number of iterations. 
\end{proof}

\paragraph{The {\algBal} - Proof of \cref{theorem:lower_bound_bal_EA_path}:}
The proof of \cref{theorem:lower_bound_bal_EA_path} works analogously to that of \cref{theorem:lower_bound_EA_path}.
The only substantial difference is in
the transition probabilities  in~\cref{lemma:endpoint_transitions}. In particular, the \algBal satisfies 
\begin{align*}
	P[l(X_{t+1}) &= l(X_{t}) + 2 | X_t] = 1/2(1/n^2 + 1/n), \text{ and }\\
	P[r(X_{t+1}) &= r(X_{t}) - 2 | X_t] = 1/2(1/n^2 + 1/n).
\end{align*}
Clearly $1/2(1/n^2 + 1/n) \le 1/n$, which gives a lower bound of $\bigOmega{n^3}$ for the \alg1+1.

\paragraph{Experimental Evidence for Long Bad Paths:}
So far we have argued that given a feasible initial configuration with exactly one bad path of length in $\bigTheta{n}$, the \alg1+1 requires with constant probability $\bigOmega{n^4}$ iterations to reach the optimum.
Analogously, we have shown that the \algBal{} requires $\bigOmega{n^3}$ steps.

We are now going to present experimental evidence that, starting from a uniformly chosen configuration, the \alg1+1 and the \algBal{} both reach a state $X \in \{0, 1\}^n$ with  $f(X) = \opt + 1$ with high probability. 
Once such a state is reached, our experiments suggest that the bad path of $X$ has linear length with at least constant probability.

We did experiments on paths of odd length $n \in \{51 + 10 \cdot k \mid k \in \N \cap [0, 15]\}$ with $100$ runs for each $n$ and each algorithm.
In every iteration, the algorithm starts with a uniformly random initial configuration. 
Once a state $X$ with $f(X) = \opt + 1$ is reached, we record the length of the bad path in $X$, divided by $n$. Since $X$ has only one vertex more than the optimal cover, there will be only one bad path, and all other vertices will have their optimal value.

In the case that the algorithm never reaches such a state (i.e. it jumps directly from fitness level $\ell>1$ to the optimum), the length of the bad path is set to zero, which is a worst-case. Interestingly, this behaviour wasn't observed in even a single simulation, which suggests that such a jump is difficult for either algorithm to perform. 

\begin{figure}
	\begin{subfigure}{0.4999\textwidth}
		\includegraphics[width=\textwidth]{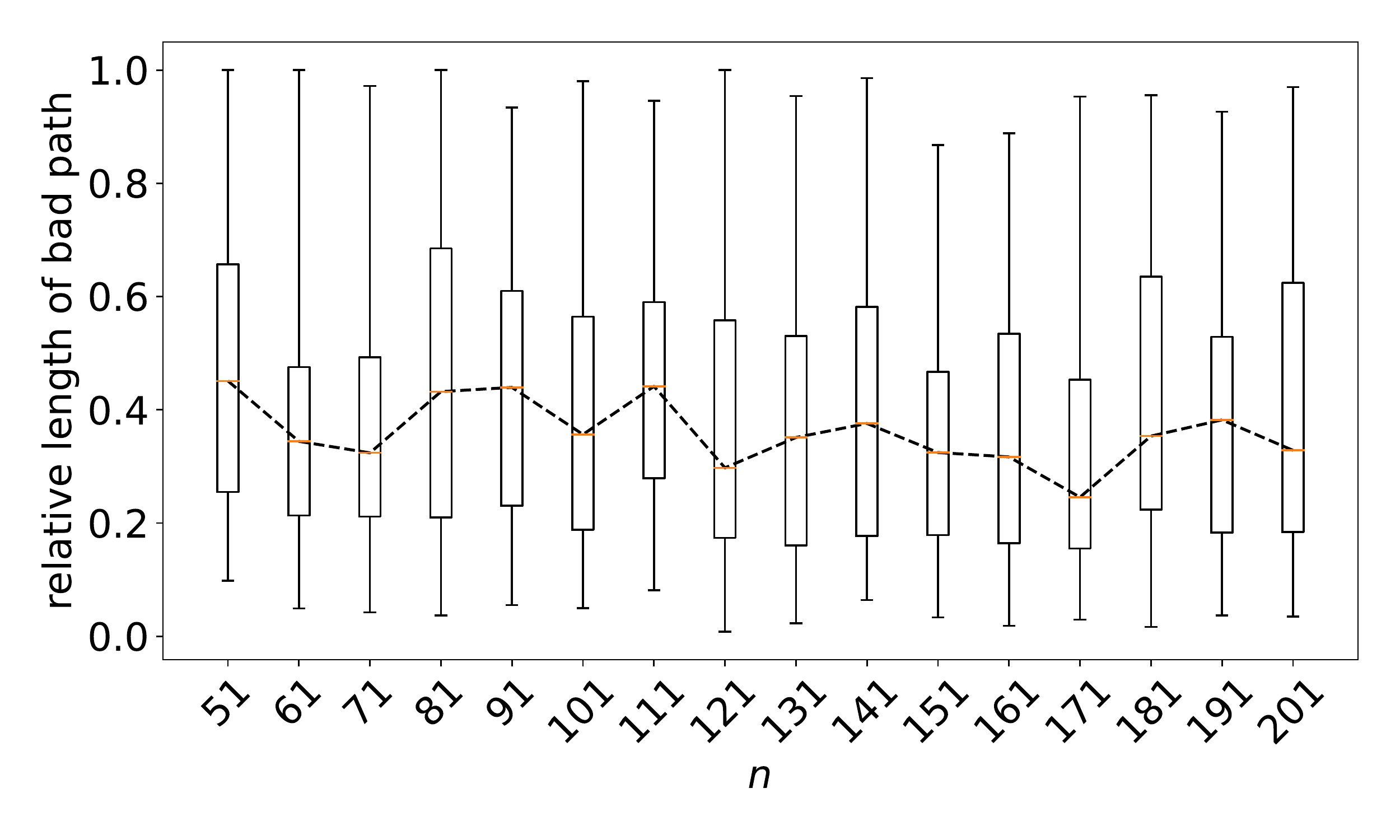}
		\caption{\alg1+1}
	\end{subfigure}
	\begin{subfigure}{0.4999\textwidth}
		\includegraphics[width=\textwidth]{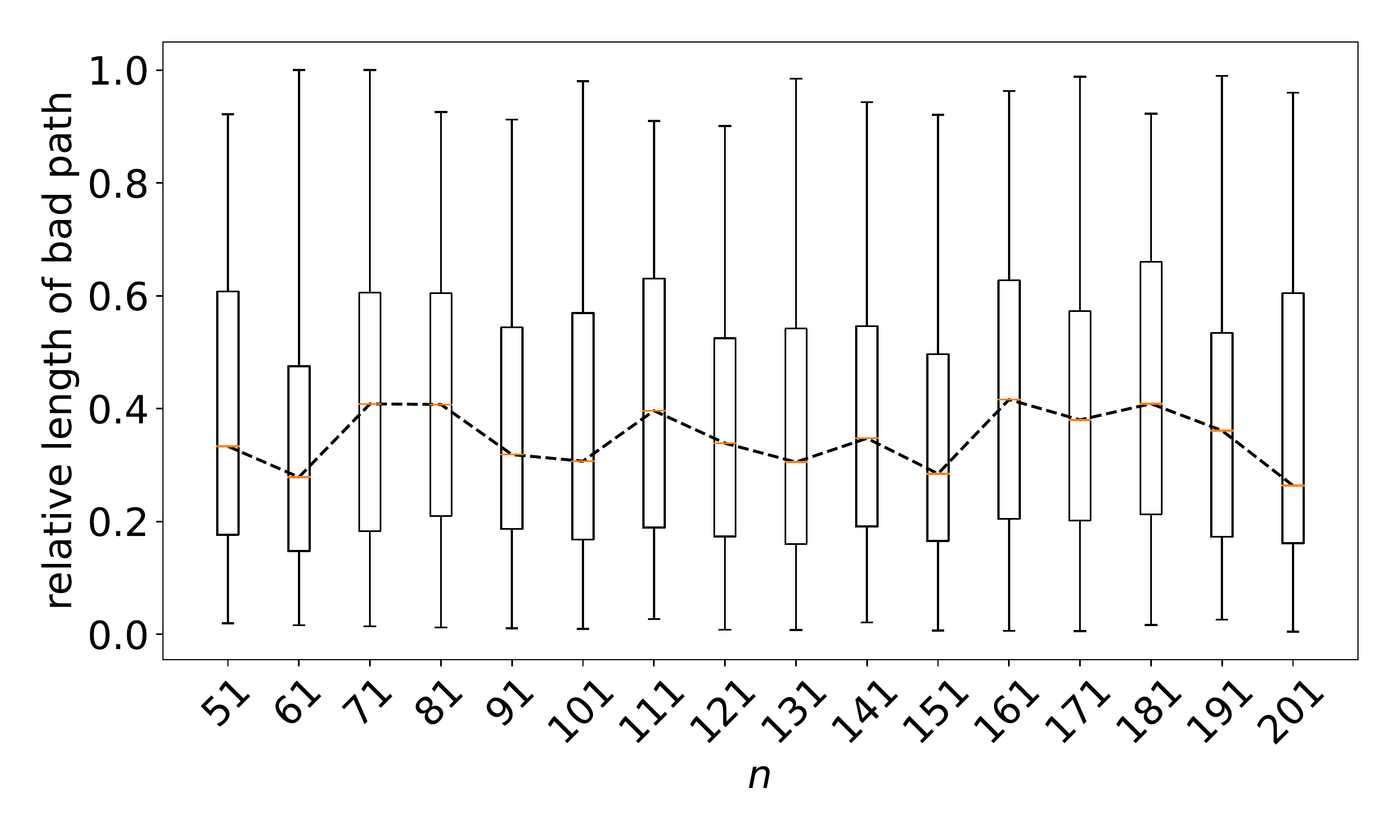}
		\caption{\algBal}
	\end{subfigure}
	\caption{Relative length of the bad path at last fitness level for \alg1+1 and \algBal}
	\label{fig:EA_bad_paths}
\end{figure}

The results are given in \cref{fig:EA_bad_paths}.
Each box indicates span from the first to the third quartile (i.e.~the inner $50\%$ of the data) and the whiskers indicate minimum and maximum values that were observed.
The result for both algorithms look very alike.
Apart from minor deviations, most of the iterations resulted in similar relative lengths of the bad paths, independent of $n$.
The medians seem to fluctuate around a value of $1/3$ and never drop below $1/5$,
which behavior is consistent throughout all tested values of $n$.

Our experiments suggest that, with constant probability, both algorithms enter the last fitness level with a bad path of linear length.
Therefore, when starting from a uniformly random initial configuration, with constant probability a state is reached to which Theorems \ref{theorem:lower_bound_EA_path} and \ref{theorem:lower_bound_bal_EA_path} apply.
This suggests that, even when starting from a uniform random initialization, the expected number of iterations to find the optimum are in $\bigOmega{n^4}$ for the \alg1+1 and in $\bigOmega{n^3}$ for the \algBal, due to the appearance of such linear bad paths.

	\section{Vertex Cover on Complete Bipartite Graphs}\label{sec: vertex cover bipartite}

A common instance class of interest when it comes to evolutionary algorithms for the vertex cover problem are complete bipartite graphs.
In this section, we investigate basic properties of the \algBal{} in this setting.
Specifically, for $L, R \in \N \setminus \{0\}$, let $K_{L, R} = (V, E)$ be a graph, such that $V$ can be partitioned into sets $V_L$ and $V_R$ with $\size{V_L} = L$ and $\size{V_R} = R$, and $E = \{\{v, u\} \mid v \in V_L \text{ and } u \in V_R\}$.
We call $V_L$ the left partition and $V_R$ the right partition.
A bit string $X \in \{0, 1\}^{L+R}$ represents a feasible solution for the vertex cover problem on $K_{L, R}$ if and only if $V_L \subseteq X$ or $V_R \subseteq X$.

Now assume $R/L = c$ for some $c > 1$,
and note that in this case $X = V_L$ is the unique optimum of the vertex cover problem on $K_{L, R}$.
On the other hand, the state $X = V_R$ represents a local optimum with respect to the fitness function $f$.
Once this local optimum is reached, $X$ must flip at least $2 L$ bits in one step to escape. 
This causes the \alg1+1 to need an exponential number of iterations in expectation, and local search heuristics like Random Local Search (RLS) to fail completely.
The probability of reaching the optimal solution in polynomial time is usually closely related to the probability of never getting too close to that local optimum, which makes complete bipartite graphs to an interesting instance class to study.

In this section, we specifically consider the following setting.
We assume the ratio $c \coloneqq R/L > 0$ to be fixed and investigate the probability for finding the optimum in polynomial time asymptotically in $L$.
Our main result is that the probability of not finding the optimum in polynomial time decays exponentially in $L$ as long as $c>2$.
As \Cref{table:complete_bipartite_graph} shows, this is quite different from the RLS, where the probability is constant for any fixed $c$, and the normal \alg1+1, where it only decays polynomialy in $L$. 
Note all asymptotic behavior here is stated in terms of $L$.
This can be easily translated to asymptotics in terms of the absolute number of vertices $n = L + R = (c+1)L$.
In particular, as long as the ratio $c$ is considered a fixed constant, this only changes linear factors, leading to the results that are presented in the abstract. 

\begin{table}[h]
	\centering
	\begin{tabular}{c|c|c}
		\textbf{RLS}                                                                                
		& \textbf{\alg1+1}	
		& \textbf{\algBal}\\ 
		\hline
		\multicolumn{1}{l|}{
			\begin{tabular}[c]{@{}c@{}}
				$\le \frac{c}{c+1}$ \\ 
				\cite[Theorem $4$]{friedrich2010approximating}
		\end{tabular}} 
		& \multicolumn{1}{l|}{
			\begin{tabular}[c]{@{}c@{}}
				$\le 1 - \frac{1}{\text{poly}(L)}$ \\
				\cite[Theorem $5$]{friedrich2010approximating}
				\footnotemark
		\end{tabular}} 
		& \multicolumn{1}{l}{
			\rule{0pt}{6ex} 
			\begin{tabular}[c]{@{}c@{}}
				$\ge 1 - 2^{-\bigOmega{L^{\xi}}}$ 
				for all $0 \le \xi <1$ if $c > 2$\\
				see \Cref{theorem:complete_bipartite_balanced_ea_runtime} 
		\end{tabular}} 
		\rule{0pt}{6ex} 
	\end{tabular}
	\caption{Probability of finding the minimum vertex cover on $\boldsymbol{K_{L, R}}$ with $\boldsymbol{c \coloneqq R/L > 1}$ in time $\poly{\boldsymbol L}$.}
	\label{table:complete_bipartite_graph}
\end{table}%
\footnotetext{In fact it is part of the proof of Theorem $5$ in \cite{friedrich2010approximating}, but not part of the main statement.}

The main technical result that we use to prove this bound on the number of iterations can informally be stated as follows. 
As long as $V_R$ is more than twice as big as $V_L$, the probability of selecting at least $R-L$ vertices in $V_R$ before selecting all vertices in $V_L$ decays exponentially in the size of $L$.
\begin{restatable}{lemma}{completeBipartiteBalancedEA}
	\label{theorem:complete_bipartite_balanced_ea}
	Let $K_{L, R}$ be a complete bipartite graph with vertex partitions $V_L$ of size $L$ and $V_R$ of size $R$, and set $c \coloneqq R/L$.
	Let $X_t \in \{0, 1\}^{L+R}$ be the sequence of states of the \algBal{} with fitness function $f$ on $K_{L, R}$ and let $X_0$ be chosen uniformly at random.
	Define $T_L = \inf \{t \in \N \mid  \size{V_L \cap X_t} = L\}$ and let $T'_R = \inf \{t \in \N \mid \size{V_R \cap X_t} \ge R - L\}$.
	If $c > 2$, then, for all positive constants $\xi < 1$ and $L$ sufficiently large, it holds that $\Pr{T_L \ge T'_R} \le 2^{- \bigOmega{L^{\xi}}}$.
\end{restatable}

The following probabilistic bound on the number of iterations for finding the minimum vertex cover on $K_{L, R}$ can be derived from \Cref{theorem:complete_bipartite_balanced_ea}.
\begin{restatable}{theorem}{completeBipartiteBalancedEARuntime}
	\label{theorem:complete_bipartite_balanced_ea_runtime}
	Let $K_{L, R}$ be a complete bipartite graph with vertex partitions $V_L$ of size $L$ and $V_R$ of size $R$, and set $c \coloneqq R/L$.
	If $c>2$ then, for all positive constants $\xi < 1$, it holds that the \algBal{} with fitness function $f$, starting from a uniformly random initial configuration, requires at most $\bigO{(c+1)L^2 \log((c+1)L)}$ iterations to find the minimum vertex cover of $K_{L,R}$ with probability at least $1 - 2^{- \bigOmega{L^{\xi}}}$.
\end{restatable}

\begin{proof}
	We prove this statement in two steps.
	At first, we argue that, with sufficiently high probability, we reach a state where all vertices in $V_L$ are selected and more than $L$ vertices from $V_R$ are not selected in at most $\bigO{(c+1)L^2 \log((c+1)L)}$ iterations.
	In the second step, we argue that, once such a state is reached, we go to the optimum in at most $\bigO{(c+1)L^2 \log((c+1)L)}$ iterations with high probability.
	
	For the first part, let $X_t$ denote the states of the \algBal{} and let $T_L$ and $T'_R$ be defined as in \Cref{theorem:complete_bipartite_balanced_ea}.
	Formally, we want to lower bound $\Pr{T_L \le \tau \text{ and } T_L < T'_R}$ for some suitably chosen $\tau \in \bigO{(c+1)L^2 \log((c+1)L)}$.
	Let $T = \inf\{t \in \N \mid X_t \text{ is feasible}\}$ and observe that, if $T_L < T'_R$, then $T = T_L$.
	Thus, we have
	\[
		\Pr{T_L \le \tau \text{ and } T_L < T'_R} = \Pr{T \le \tau \text{ and } T_L < T'_R}.
	\]
	Using De Morgan's laws and union bound, we get
	\begin{align*}
		\Pr{T \le \tau \text{ and } T_L < T'_R} 
		&= 1 - \Pr{T > \tau \text{ or } T_L \ge T'_R}\\
		&\ge 1 - \Pr{T > \tau} - \Pr{T_L \ge T'_R} .  
	\end{align*}
	By \Cref{theorem:complete_bipartite_balanced_ea}, we have $\Pr{T_L \ge T'_R} \le 2^{- \bigOmega{L^{\xi}}}$.
	Setting $\tau = 4 \eulerE L (L+R) \left( \ln(L+R) + 1/2\right)$, \Cref{lemma::feasible} yields  $\Pr{T \ge \tau} \le 2^{-L}$.
	Thus, we get that the \algBal{} reaches a state with all vertices in $V_L$ selected and more than $L$ vertices in $V_R$ are not selected in at most $\tau \in \bigO{(c+1)L^2 \log((c+1)L)}$ iterations.
	
	Now, assume we start in a state $X_0$ with $V_L \subseteq X_0$ and $\size{X_0 \cap V_R} < R - L$.
	Note that adding vertices from $V_R$ requires removing at least as many vertices from $V_L$.
	However, to remain feasible, removing any vertex from $V_L$, requires adding all vertices from $V_R$ that are not selected yet.
	As there are at lest $L+1$ such vertices and at most $L$ vertices in $V_L$ that could be removed, the \algBal{} can never add a vertex from $V_R$ or remove a vertex from $V_L$ in such a state.
	Thus, all we need to bound is the required time to remove all vertices in $X_0 \cap V_R$ to also bound the time for reaching the optimum.
	First, observe that for each fixed vertex in $X_0 \cap V_R$, the probability of not being removed after $\tau' \in \N$ steps is at most $\big(1 - \frac{1}{L + R}\big)^{\tau'} \le \eulerE^{- \frac{\tau'}{L + R}} = \eulerE^{- \frac{\tau'}{(c+1)L}}$.
	Thus, using union bound, the probability for not reaching the optimum after $\tau'$ steps is at most 
	$$
		\size{X_0 \cap V_R}\eulerE^{- \frac{\tau'}{(c+1)L}} < (R-L) \eulerE^{- \frac{\tau'}{(c+1)L}} = (c - 2)L \eulerE^{- \frac{\tau'}{(c+1)L}}.
	$$
	Choosing $\tau' = (c+1) L \left(\ln((c-2)L) + L\right) \in \smallO{(c+1)L^2 \ln((c+1)L)}$, this probability is upper bounded by $\eulerE^{-L}$.
	
	Observe that, if the first phase needs $\tau$ iterations and the second phase needs $\tau'$ iterations, then the total number of iterations for finding the optimum is $\tau + \tau' \in \bigO{(c+1)L^2 \ln((c+1)L)}$.
	The probability that at least one of both phases fails is at most $2^{-\bigOmega{L^{\xi}}} + \eulerE^{-L} \le 2^{-\bigOmega{L^{\xi}}}$, which concludes the proof.
\end{proof}

We proceed by proving \Cref{theorem:complete_bipartite_balanced_ea}.
To simplify notation, we define the processes $L_t = V_L \cap X_t$ and $R_t = V_R \cap X_t$.
A central step in our proof of \Cref{theorem:complete_bipartite_balanced_ea} is to argue that $T'_R$ is super polynomial with high probability.
Once this is done, it remains to show that, with sufficiently high probability, all vertices in $V_L$ are selected beforehand.

To prove the first part, we apply the following negative drift theorem.

\begin{theorem}[Oliveto et.~al~\cite{oliveto2011simplified}]
	\label{thm::negative_drift}
	Let $(Y_t)_{t \in \N}$ be random variables over an interval $I$ adapted to a filtration $(\mathcal{F}_t)_{t \in \N}$, and let $b \in I$ such that $Y_0 \geq b$.
	Suppose that there exists an interval $[a, b] \subseteq I$, two constants $\delta,\varepsilon>0$, and, possibly depending on $d := b - a$, a function $r(d)$ satisfying $1 \leq r(d) \in o(d/\log(d)$ such that, for all $t \in N$, it holds that
	\begin{enumerate}[label=(\alph*)]
		\item 
		\label{thm::negtive_drif:bullet_a} 
		$\E{(Y_{t+1}-Y_t) \cdot \1_{\{a < Y_t < b\}}}[\mathcal{F}_t] \geq \varepsilon \cdot \1_{\{a < Y_t < b\}}$, and
		\item 
		\label{thm::negtive_drif:bullet_b}
		for all $j \in \N$, it holds that
		$$ \Pr{\abs{Y_{t+1} - Y_t} \cdot \1_{\{Y_t > a\}} \geq j \cdot \1_{\{Y_t > a\}} \mid \mathcal{F}_t} \leq \frac{r(d)}{(1+\delta)^j}.$$
	\end{enumerate}
	
	Then there is a constant $z > 0$ and a function $m(d) \in \Omega(d/r(d))$ such that, for the hitting time $T := \inf\{t \in \N \mid Y_t \leq a\}$, it holds that	
	$$\Pr{T \leq 2^{\frac{zd}{r(d)}}} = 2^{-m(d)}.$$
\end{theorem}
 
An obvious choice would be apply the above theorem directly to $|R_t|$. However, this would lead to rather complicated transition probabilities.
Instead, we define the following slightly modified process $S_t \in \{0,1\}^R$ based on $X_t$:
\begin{itemize}
	\item Initially, $S_0 = R_0$.
	\item Whenever $X_t$ attempts to flip bits for some set of vertices $A \subseteq V$, $S_t$ flips all bits in $(V_R \cap A) \setminus S_t$ (in set notation, this is $S_{t+1} = S_{t} \cup (V_R \cap A)$). 
	\item Whenever $X_t$ attempts a balanced flip with initial vertex $v \in S_t$, the corresponding bit in $S_t$ is set to $0$ (i.e., $S_{t+1} = S_t \setminus \{v\}$).
\end{itemize}

The following relationship between $R_t$ and $S_t$ will come in handy throughout our analysis.
\begin{restatable}{lemma}{RsubsetS}
	\label{lemma:R_subset_S}
	Let $T_L$, $T'_R$ and $S_t$ be defined as above.
	Then $R_t \subseteq S_t$ for all $t \le \min\{T_L, T'_R\}$.
\end{restatable}
\begin{proof}
	We prove the statement via induction over $t$.
	For the base case $t=0$, note that $S_0 = R_0$ by definition.
	Consider some $0 < t \le \min\{T_L, T'_R\}$ and assume $R_{t-1} \subseteq S_{t-1}$.
	If the \algBal{} attempts to flip the bits that correspond to a the set of vertices $A \subseteq V$ in iteration $t$, then $R_{t} \subseteq R_{t-1} \cup (A \cap V_R) \subseteq S_{t-1} \cup (A \cap V_R) = S_{t}$.
	
	Now, assume the \algBal{} attempts to do a balanced flip in iteration $t$ and selects some $v \in V$ as starting vertex.
	Note that $t \le T'_R$ implies that $t-1 < T'_R$ and therefore $\abs{R_{t-1}} < R-L$, which means that a balanced flip can never add a vertex in the right partition.
	If $v \in V_L$, then either a vertex in the right partition is removed or $R_t = R_{t-1}$.
	In any case, $R_t \subseteq R_{t-1} \subseteq S_{t-1} = S_{t}$.
	For $v \in V_R$, we distinguish between three different cases.
	If $v \in V_R \setminus S_{t-1}$, then also $v \in V_R \setminus R_{t-1}$ and nothing happens.
	On the other hand, if $v \in R_{t-1}$, then $t \le T_L$ implies $t-1 < T_L$ and $L \setminus L_t \neq \emptyset$.
	Thus, the bit corresponding to $v$ gets flipped to $0$ and a vertex from the left partition is added to the solution.
	It holds that $R_t = R_{t-1} \setminus \{v\} \subseteq S_{t-1} \setminus \{v\} = S_t$.
	Finally, assume that $v \in S_{t-1} \setminus R_{t-1}$.
	Note that the existence of such a vertex $v$ implies $R_{t-1} \subseteq S_{t-1} \setminus \{v\}$.
	In this case, the \algBal{} does nothing and $v$ is removed from $S_{t-1}$.
	Thus, we have $R_t = R_{t-1} \subseteq S_{t-1} \setminus \{v\} = S_t$.
	In all cases, $R_t \subseteq S_t$, which concludes the proof. 
\end{proof}

We are aiming to use the stopping time $T_S = \inf \{t \in \N \mid \size{S_t} \ge R - L\}$ as a lower bound for $T'_R$.
The following lemma justifies this.

\begin{restatable}{lemma}{TRdominatesTS}
	\label{lemma:TR_dominates_TS}
	For $T'_R$ and $T_S$ as above it holds that $T'_R \ge T_S$.
\end{restatable}
\begin{proof}
	We distinguish between two cases.
	If $T_L \ge T'_R$, then \Cref{lemma:R_subset_S} implies that $R_t \subseteq S_t$ for all $t \le T'_R$.
	Thus, $\size{R_t} \le \size{S_t}$ for all such $t$ and especially $R - L \le \size{R_{T'_R}} \le \size{S_{T'_R}}$, which implies $T'_R \ge T_S$.
	Now, assume $T_L < T'_R$.
	Note that this implies that $X_{T_L}$ is feasible and has a fitness value of $f\left(X_{T_L}\right) < R$.
	Observe that increasing the number of selected vertices in the right partition would require removing vertices from the left partition.
	However, to remain feasible, all vertices from the right partition must be added, resulting in a fitness value of at least $R$.
	Thus, we have for all $t > T_L$ that $\size{R_t} \le \size{R_{T_L}} < R - L$.
	Therefore, $T'_R = \infty$ and $T_S \le T'_R$ holds trivially. 
\end{proof}

Based on \Cref{lemma:TR_dominates_TS}, the following statement is essentially derived by applying a negative drift argument to $\size{S_t}$.

\begin{restatable}{lemma}{TRlowerBound}
	\label{lemma:TR_lower_bound}
	Let $T'_R$ be defined as above.
	If $c > 2$, then, for every positive constant $\xi < 1$ and all $L \ge \left(\max\left\{2, 4/(c-2)\right\}\right)^{\frac{1}{1-\xi}}$, there is a constant $\alpha > 0$ and a function $g(L) \in \Omega(L^{\xi})$ such that
	\[
		\Pr{T'_R \le 2^{\alpha L^{\xi}}} \le 2^{-g(L)} ,
	\] 
	given the initial configuration $X_0$ is chosen uniformly at random.
\end{restatable}

\begin{proof}
	We start by assuming that the initial configuration is chosen such that $\size{R_0} \le R - L - L^{\xi}$. 
	By \Cref{lemma:TR_dominates_TS} it is sufficient to prove the statement for $T_S$ instead.
	To simplify notation, set $s_t = \size{S_t}$.
	Observe that $s_t$ has the following transition probabilities:
	\begin{itemize}
		\item $\Pr{s_{t+1} = s_{t} - 1}[s_{t}] = \frac{1}{2} \cdot \frac{s_{t}}{L+R}$
		\item for all $j \in \N \cap [1, R-r_t]$:\\
		$\Pr{s_{t+1} = s_{t} + j}[s_{t}] = \frac{1}{2} \cdot \binom{R-r_{t}}{j} \cdot \left(\frac{1}{L+R}\right)^j \cdot \left(1 - \frac{1}{L+R}\right)^{R - r_t - j}$
		\item with the remaining probability mass, $s_{t+1} = s_{t}$
	\end{itemize}
	Note that the second case (i.e., $s_t$ increases) can be described more intuitively as follows.
	With a probability of $\frac{1}{2}$, we draw a binomial random variable $Q$ with $R-r_t$ trials and success probability $\frac{1}{L+R}$, and $s_{t+1} = s_{t} + Q$.
	
	To bound $T_S$, we apply \Cref{thm::negative_drift} to the transformed stochastic process $Y_t = R - L - s_t$ with its natural filtration $\filtration[t]$ and stopping time $T_Y = \inf \{t \in \N \mid Y_t \le 0\} = T_S$.
	Note that, for $s_t \le R - L - L^{\xi}$, we have $Y_t \ge L^{\xi}$.
	To this end, we argue that, for any fixed constant $c > 2$ and all $L \ge (\frac{4}{c-2})^{\frac{1}{1-\xi}}$, the requirements \ref{thm::negtive_drif:bullet_a} and \ref{thm::negtive_drif:bullet_b} of \Cref{thm::negative_drift} are satisfied for $a = 0$, $b = L^{\xi}$, $r(d) = \eulerE^{1/3} \ge 1$, $\delta = \eulerE^{1/3} - 1 > 0$ and $\epsilon = \frac{c - 2}{4 (c + 1)} > 0$.
	
	\paragraph{\Cref{thm::negative_drift} condition \ref{thm::negtive_drif:bullet_a}:}
	Note that this condition is trivially satisfied for $Y_t \le a = 0$ and $Y_t \ge b = L^{\xi}$ due to the indicator functions.
	To simplify notation, we assume from now on that $Y_t \in (0, L^{\xi})$ and omit those indicator functions.
	By the definition of $Y_t$, it holds that
	\begin{align*}
		\E{Y_{t+1} - Y_t}[\filtration[t]] 
		&= -\E{s_{t+1} - s_t}[\filtration[t]] \\
		&= - \frac{1}{2} \cdot \left(\frac{R - s_t}{R + L} - \frac{s_t}{R + L}\right)\\
		&= \frac{2 s_t - R}{2(R + L)}.
	\end{align*}
	Substituting $s_t = R - L - Y_t$ and $R = c L$, and upper bounding $Y_t$ by $L^{\xi}$ yields 
	\[
	\E{Y_{t+1} - Y_t}[\filtration[t]] = \frac{R - 2 L - 2 Y_t}{2(R + L)} \ge \frac{(c-2)}{2(c+1)} - \frac{L^{\xi - 1}}{c+1} .
	\]
	Thus, for $L \ge \left(\frac{4}{c-2}\right)^{\frac{1}{1-\xi}}$ we have $\E{Y_{t+1} - Y_t}[\filtration[t]] \ge \frac{c - 2}{2(c+1)} = \epsilon$.
	
	\paragraph{\Cref{thm::negative_drift} condition \ref{thm::negtive_drif:bullet_b}:}
	Again, we omit the indicator functions and assume $Y_t > a = 0$.
	We do a case distinction based on $j \in \N$.
	First, observe that for $j \in \{0, 1\}$ it trivially holds that 
	\[
	\Pr{\abs{Y_{t+1} - Y_t} = j \le 1} \le \frac{\eulerE^{1/3}}{\eulerE^{j/3}} = \frac{r(d)}{(1+\delta)^j}
	\]
	Thus, the condition is satisfied.
	For $j \ge 2$, note that such jumps can only appear when $Y_t$ decreases.
	Therefore, we focus on the distribution $Y_t - Y_{t+1} = s_{t+1} - s_t$.
	As discussed earlier, given $s_t$, $s_{t+1} - s_t$ is dominated by a binomial random variable $Q$ with $R-s_t$ trials and success probability $\frac{1}{L+R}$.
	$Q$ in turn is dominated by a binomial random variable $Q'$ with the same success probability but $L+R$ trials.
	By Chernoff's bound (see Theorem $1.10.1$ in \cite{Doerr2019}) and the fact that $\E{Q'} = 1$ we obtain
	\[
	\Pr{\abs{Y_{t+1} - Y_{t}} = j}[Y_t] \le \Pr{Q' \ge j} \le \eulerE^{- \frac{(j-1)}{3}} = \frac{r(d)}{(1+\delta)^j}.
	\] 
	
	To save space, we write $\lambda$ for $R - L - L^{\xi}$.
	Applying \Cref{thm::negative_drift}, we obtain that there is a constant $z > 0$ and a function $m(d) \in \bigOmega{d}$ such that $\Pr{T_Y \le 2^{\frac{z d}{\eulerE^{1/3}}}}[R_0 \le \lambda] \le 2^{-m(d)}$.
	Noting that $d = L^{\xi}$ and setting $\alpha = \frac{z}{\eulerE^{1/3}}$ proves that, for $g'(L) = m(L^{\xi})$ we obtain
	\[
	\Pr{T'_R \le 2^{\alpha L^{\xi}}}[\size{R_0} \le \lambda]
	\le \Pr{T_S \le 2^{\alpha L^{\xi}}}[\size{R_0} \le \lambda] 
	\le 2^{-g'(L)}.
	\]
	
	Next, assume we start with an initial configuration $X_0$ that is chosen uniform at random.
	Observe that $\abs{R_0}$ follows a binomial distribution with success probability $\frac{1}{2}$ and $R$ trials.
	Next, observe that for $L \ge 2^{\frac{1}{1-\xi}}$ and $c > 2$ it holds that
	\[
	\left(1 + \frac{1}{2}\right) 
	\ge \left(2 - \frac{1}{c}\right) 
	\ge \left(2 - \frac{2 - 2 L^{\xi - 1}}{c}\right)
	= \frac{c L - L - L^{\xi}}{\frac{c}{2} L}.
	\]
	As $\E{\size{R_0}} = \frac{1}{2} R = \frac{c}{2} L$, we have $\left(1 + \frac{1}{2}\right)\E{\size{R_0}} \ge c L - L - L^{\xi}$ and by Chernoff's bound (see Theorem $1.10.1$ in \cite{Doerr2019})
	\[
	\Pr{\size{R_0} > \lambda} \le \Pr{\size{R_0} > \left(1 + \frac{1}{2}\right) \E{\size{R_0}}} \le \eulerE^{- \frac{c}{24} L} .
	\]
	By using the law of total probability we obtain
	\begin{align*}
		&\Pr{T'_R \le 2^{\alpha L^{\xi}}}\\ 
		&= \Pr{T'_R \le 2^{\alpha L^{\xi}} \text{ and } \size{R_0} \le \lambda} + \Pr{T'_R \le 2^{\alpha L^{\xi}} \text{ and } \size{R_0} > \lambda} \\
		&\le \Pr{T'_R \le 2^{\alpha L^{\xi}}}[\size{R_0} \le \lambda] + \Pr{\size{R_0} > \lambda} \\
		&\le 2^{-g'(L)} + \eulerE^{- \frac{c}{24} L} \\
		&\le 2^{-g(L)}
	\end{align*}
	for some function $g(L) \in \bigOmega{L^{\xi}}$, which proves the claim.
\end{proof}

Having \Cref{lemma:TR_lower_bound} at hand, we are ready to prove \Cref{theorem:complete_bipartite_balanced_ea} and hence our main theorem of this section (\cref{theorem:complete_bipartite_balanced_ea_runtime}).

\begin{proof}[Proof of \Cref{theorem:complete_bipartite_balanced_ea}]
	By \Cref{lemma:TR_lower_bound} we know that for $L$ sufficiently large and $c > 2$ there is a constant $\alpha > 0$ and a function $g(L) \in \Omega(L^{\xi})$ such that $\Pr{T'_R \le 2^{\alpha L^{\xi}}} \le 2^{- g(L)}$.
	
	We know that $\Pr{T_L < T'_R} \ge \Pr{T_L \le \tau \text{ and } T'_R > \tau}$ holds for every $\tau \in \N$.
	Let $T = \inf \{t \in \N \mid X_t \text{ is feasible }\}$ and observe that, if $T'_R > \tau$, then $T_L \le \tau$ if and only if $T \le \tau$.
	Consequently, we have $\Pr{T_L \le \tau \text{ and } T'_R > \tau} = \Pr{T \le \tau \text{ and } T'_R > \tau}$.
	Moreover, using union bound, we get
	\begin{align*}
		\Pr{T \le \tau \text{ and } T'_R > \tau} &= 1 - \Pr{T > \tau \text{ or } T'_R \le \tau}\\
		&\ge 1 - \Pr{T > \tau} - \Pr{T'_R \le \tau} .
	\end{align*}
	By choosing $\tau = 4 \eulerE (c+1) L^2 \left(\ln((c+1)L) + 1/2\right)$, \Cref{lemma::feasible} yields $\Pr{T > \tau} \le 2^{-L}$.
	Moreover, for $L$ sufficiently large, we have $\tau \le 2^{\alpha L^{\xi}}$.
	Thus, $\Pr{T'_R \le \tau} \le 2^{-g(L)}$ for a function $g(L) \in \Omega(L^{\xi})$.
	Consequently, $\Pr{T_L < T'_R} \ge 1 - 2^{-\bigOmega{L^{\xi}}}$ and $\Pr{T_L \ge T'_R} \le 2^{-\bigOmega{L^{\xi}}}$, which concludes the proof.
\end{proof} 

Besides \Cref{theorem:complete_bipartite_balanced_ea_runtime}, a variety of other properties that might be of independent interest can be derived from \Cref{theorem:complete_bipartite_balanced_ea}.
For example, the following corollary shows that the probability of $X$ filling the larger side before all vertices in the smaller partition are selected, and the probability to ever select the entire larger side at all, both decay exponentially in $L$ as well.
\begin{restatable}{corollary}{corrolaryCompleteBipartiteBalancedEA}
	\label{corolarry:complete_bipartite_balanced_ea}
	Consider the setting of \Cref{theorem:complete_bipartite_balanced_ea} and let $T_R = \inf \{t \in \N \mid  V_R \subseteq X_t\}$.
	If $c > 2$, then, for all positive constants $\xi < 1$ and $L$ sufficiently large, it holds that
	\begin{enumerate}[label=(\arabic*)]
		\item $\Pr{T_L \ge T_R} \le 2^{- \bigOmega{L^{\xi}}}$
		\label{corolarry:complete_bipartite_balanced_ea:bullet_1}
		\item $\Pr{T_R < \infty} \le 2^{- \bigOmega{L^{\xi}}}$ .
		\label{corolarry:complete_bipartite_balanced_ea:bullet_2}
	\end{enumerate} 
\end{restatable}
\begin{proof}
	Part \ref{corolarry:complete_bipartite_balanced_ea:bullet_1} follows trivially from $T_R \ge T'_R$.
	For part \ref{corolarry:complete_bipartite_balanced_ea:bullet_2}, note that, as soon as all vertices in $V_L$ are selected, adding vertices to the right partition requires removing at least as many vertices on the left.
	To remain feasible, however, all free vertices in the right partition must be added at once.
	If there are more than $L$ vertices in the right partition not selected, this is not possible as there are at most $L$ vertices in the left partition in total that could be removed.
	Thus, from this point on, no more vertex from the right partition can be added and $T_L < T'_R$ implies $T_R = \infty$.
\end{proof}

	\bibliography{literature}
\end{document}